\documentclass[a4paper, 11pt, reqno]{amsart} %11pt?
\usepackage[T1]{fontenc}
\usepackage[utf8]{inputenc}
\usepackage[text={6.5in,8.6in},centering]{geometry}
\usepackage[toc,page]{appendix}
\usepackage[english]{babel}
\usepackage{graphicx}
\usepackage{amsmath,mathtools, mathrsfs}
\usepackage{amsfonts,amsmath, amssymb, amsthm, dsfont}
\usepackage{hyperref}
\usepackage{tikz}
\usepackage{tikz-3dplot}
\usepackage{comment}
\usepackage{url}
\usepackage{dutchcal}
\usetikzlibrary{arrows.meta,calc,decorations.pathreplacing}
\newtheorem{theorem}{Theorem}[section]
\newtheorem{proposition}[theorem]{Proposition}
\newtheorem{lemma}[theorem]{Lemma}
\newtheorem{corollary}[theorem]{Corollary}
\newtheorem{definition}[theorem]{Definition}
\newtheorem{remark}[theorem]{Remark}
 
\newcommand{\abs}[1]{\left\lvert#1\right\rvert} 
 
 \newcommand{\set}[1]{\left\{#1\right\}}
 \title{Poincaré inequalities on hyperbolic-type spaces}
\author[]{Anna Mc Court, Federico Santagati
and Maria Vallarino}
\address[Anna Mc Court]{Dipartimento di Matematica \\ Università degli Studi di Genova\\ Via Dodecaneso 35\\ 16146
Genova\\ Italy
}
\email{anna.mccourt@edu.unige.it}
\address[Federico Santagati]{Dipartimento di Scienze Matematiche ``Giuseppe Luigi Lagrange'' \\ Politecnico di Torino\\
C.so Duca degli Abruzzi 24 \\ 10129 Torino\\ Italy}
\email{federico.santagati@polito.it}

\address[Maria Vallarino]{Dipartimento di Scienze Matematiche ``Giuseppe Luigi Lagrange'' \\ Politecnico di Torino\\
C.so Duca degli Abruzzi 24 \\ 10129 Torino\\ Italy}
\email{maria.vallarino@polito.it}
\thanks{The authors are members of the Gruppo Nazionale per l'Analisi Matematica, la Probabilità e le loro Applicazioni (GNAMPA) of the Istituto Nazionale di Alta Matematica (INdAM). The second author was 
partially supported by the grant “INdAM – GNAMPA Project”, CUP E53C25002010001 and he is grateful to Nikolaos Chalmoukis and Stefano Meda for valuable discussions on the spiderweb discretisation in the upper half-plane setting}
\begin{document}
\begin{abstract}
We establish global $L^p$-Poincar\'e inequalities, for $ p\in[1,\infty)$, on a class of nondoubling hyperbolic-type metric measure spaces. The proof relies on a discretisation of the space, which gives rise to a Gromov hyperbolic graph, called spiderweb, quasi-isometric to the original space. We prove global Poincar\'e inequalities for spiderwebs endowed with suitable measures and develop a general transference principle from discrete graphs to metric measure spaces. Combining these results yields global Poincar\'e inequalities under natural geometric and measure assumptions on the base space.
\end{abstract}
\maketitle

\section{Introduction}

Poincaré inequalities play a fundamental role in analysis on metric measure spaces, where they are closely related to Sobolev spaces, heat kernel estimates, differentiability properties of Lipschitz functions, and the study of differential operators; see, for instance, the classical references \cite{Ch, HKST, SaloffCoste1, SaloffCoste2, S} and the references therein.
While local Poincar\'e inequalities are known to hold under fairly
general geometric assumptions (see e.g., \cite[Sect. 5.6.3]{SaloffCoste2} and \cite{BPV}), global inequalities are considerably more delicate,
since their validity depends on the large-scale geometry of the underlying
space.

The purpose of this paper is to establish global Poincar\'e inequalities
on a large class of locally doubling but globally nondoubling hyperbolic-type metric measure spaces through a
discrete approximation procedure.

Recall that a metric measure space $(M,d_M,\mu)$ supports a local
$L^p$-Poincar\'e inequality if for every $R>0$ there exists a constant
$C_R>0$ such that
\begin{align}\label{poinc:loc}
\int_B |f-f_B|^p\,\mathrm{d}\mu
\le
C_R\,\operatorname{rad}(B)^p
\int_{\lambda B} g_f^p\,\mathrm{d}\mu,
\end{align}
for every locally integrable function $f$ and ball $B\subset M$ with radius $\operatorname{rad}(B)\le R$, where $f_B=\mu(B)^{-1}\int_B f\,\mathrm d\mu$ denotes the average of
$f$ over $B$, $\lambda\geq1$ is a fixed dilation factor, and $g_f$ is
an upper gradient of $f$; see \cite[Sect.~7.22]{He} for a precise definition. The inequality is called {\it weak} if $\lambda>1$, and {\it strong} if $\lambda=1$.

If the constants $C_R$ may be chosen independently of $R$, then we say
that $(M,d_M,\mu)$ supports a global $L^p$-Poincar\'e inequality. In the literature, when no restriction on the scale is imposed, the adjective
``global'' is sometimes omitted, and one simply says that the space supports an
$L^p$-Poincar\'e inequality.
If \eqref{poinc:loc} holds only for balls of radius at most $R_0$, we say that the local Poincaré inequality holds up to scale $R_0$.

Let $(N,d_N, \mu_N)$ be a geodesic doubling metric measure space and consider $S=N\times \mathbb R_{+}$ where $\mathbb R_{+}=(0,\infty).$
We equip $S$ with the distance
\[
d_S((x,y),(x',y'))
=
C_N \operatorname{arcosh}
\left(
1+
\frac{
d_N(x,x')^2+(y-y')^2
}
{2yy'}
\right) \qquad \forall (x,y),(x',y')\in S,
\]
and the measure $$\mathrm{d}\rho(x,a)=\mathrm{d}\mu_N(x) \mathrm{d}a/a,$$
where $C_N>0$ is a suitable normalising parameter, introduced to obtain a 1-quasi-isometric relation with a discrete model of the space. Its precise value is irrelevant, since global Poincaré inequalities are invariant under rescaling of the metric (up to the multiplicative constants appearing in the inequality). We shall refer to $S$ as hyperbolic-type space and investigate global Poincaré inequalities on it.

The main contributions of the paper are threefold. First, we develop analysis on spiderwebs rooted at infinity (see Definition \ref{def:sw}), a
natural horospherical variant of the spiderwebs introduced in
\cite{CMS}. We introduce quasi-flow measures on such graphs, extending the class of flow measures previously considered on trees  in \cite{LSTV21}, and prove global $L^p$-Poincaré inequalities for a suitable class of spiderwebs. Second, we introduce a new discretisation of hyperbolic-type spaces, yielding a multiscale spiderweb $\mathcal{SW}_N$ associated with a dyadic cube decomposition of $N$. This graph is shown to be Gromov hyperbolic and 1-quasi-isometric to the hyperbolic-type space. Third, we establish a general transference theorem, of independent interest, which allows global Poincaré inequalities to be transferred from graphs to metric measure spaces under natural metric and measure compatibility assumptions. Beyond the applications proved here, both the spiderweb construction and the transference theorem provide tools that may be useful in the study of other analytic problems.

Combining the geometric properties of $\mathcal{SW}_N$, the global Poincar\'e
inequality on spiderwebs, and the transference principle, we obtain our
main result.

\begin{theorem}\label{main:theorem}
Let $(N,d_N,\mu_N)$ be a geodesic doubling metric measure space and let $S=N\times\mathbb R_{+}$ be the hyperbolic-type space equipped with the metric $d_S$ and the measure $\mathrm{d}\rho(x,a)=\frac{\mathrm{d}\mu_N(x)\,\mathrm{d}a}{a}.$
Then, for every $p \in [1,\infty)$ such that $N$ supports a global $L^p$-Poincar\'e inequality, $S$ also supports a global $L^p$-Poincar\'e inequality.

Furthermore, there exist positive constants $C_p$, $\mathcal{C}$, and $\mathcal{R}_0$ such that
\[
\int_{B_S(x_0,R)}
|f-f_{B_S(x_0,R)}|^p\,\mathrm{d}\rho
\le
C_p R^p
\int_{B_S(x_0,R+\mathcal{C})}
|g_f|^p\,\mathrm{d}\rho
\]
for every $x_0\in S$, every $R\ge \mathcal{R}_0$, every locally integrable function
$f$, and every upper gradient $g_f$ of $f$.
\end{theorem}
The global Poincaré inequality on $N$ required in Theorem \ref{main:theorem} is used only to deduce a local Poincaré inequality on $S$ up to a scale large enough (see Lemma \ref{lem:localPI}). Therefore, the assumption in Theorem \ref{main:theorem} can be weakened by directly requiring such a local Poincaré inequality on $S$.
Theorem \ref{main:theorem} provides a new family of locally doubling, globally nondoubling metric measure spaces with exponential volume growth which support global Poincaré inequalities.
To the best of our knowledge our main result  is new even in the model case of the \(ax+b\) group  equipped with a right Haar measure  (see for example \cite{BPV} for a local Poincaré inequality in this setting). Although locally doubling, these spaces can be highly nonhomogeneous at large scales. In particular, the measures of unit balls need not be uniformly comparable: depending on the vertical coordinate, balls of a fixed radius may have arbitrarily small or arbitrarily large measure. Combined with the exponential growth of volume, this places these spaces well outside the classical framework in which global Poincaré inequalities are usually established.
Different varieties of global Poincaré inequalities have been obtained in doubling metric measure spaces or spaces of finite measure under suitable geometric assumptions \cite{BBS, BP, CoulhonSaloffCoste1993, GL,  HajlaszKoskela2000,HeinonenKoskela1998, Jerison, RussSire}. Outside this framework, and particularly in spaces with exponential volume growth, considerably fewer examples are known. For instance, on the classical $n$-dimensional hyperbolic half-space $\mathbb R^{n-1} \times \mathbb R_+$ equipped with its Riemannian volume measure $\mathrm{d}x\mathrm{d}a/a^n$ (equivalently, a left Haar measure on the $ax+b$ group), global Poincaré inequalities do not hold \cite[Sect. 3.1]{BPV}.  Previously known examples of global Poincaré inequalities on spaces of exponential growth are those associated with flow measures on trees established in \cite[Th. 4.1]{LSTV23}. 
Observe that when $N=\mathbb R^{n-1}$ and $C_N=1$, $d_S$ coincides with the
classical hyperbolic metric on the upper half-space. In this context, when $\mu_N$ is the Lebesgue measure, $\rho$ is a right Haar measure of the $ax+b$ group. More generally, when $N$ is a stratified group endowed with a
Carnot--Carathéodory distance, the metric $d_S$ coincides, up to the
multiplicative constant $C_N$, with the canonical sub-Riemannian distance on
the solvable extension $N\rtimes\mathbb R_+$ (see Proposition~2.7 in
\cite{MartiniOttazziVallarino2018}).
 Our theorem applies, for instance, when $N$ is a doubling Lie group equipped with the Carnot--Carathéodory distance
associated with a Hörmander system of vector fields \cite[Th. 3.1]{BPV}. In the special case of $NA$ groups, the measures $\mathrm{d}\rho(x,a)=\mathrm{d}\mu_N(x) \mathrm{d}a/a$ with $\mu_N$ doubling are exactly the measures whose densities are constant along the
integral curves of the vertical vector field $X_0$, i.e. the
$X_0$-flow measures considered in \cite[Def. 2.2]{DLMV}.

The strategy of approximating a metric measure space by a suitable graph and transferring analytic properties between the discrete and continuous settings goes back to the work of Kanai \cite{K} and was systematically developed by Coulhon and Saloff-Coste \cite{CS-C}. More recently, Maity and Neelakantan \cite{MN} used discretisation and transference techniques to establish uniform Poincaré inequalities under suitable volume growth assumptions.  Our construction is related in spirit to the theory of hyperbolic
fillings of doubling metric spaces \cite{BonkSaksman, BonkSchramm2000}, but our approach differs substantially in both the geometric setting and the objective. Instead of discretising a metric space by a net, we construct a spiderweb arising from a dyadic cube decomposition of the horizontal layer and use its hyperbolic structure to establish global Poincaré inequalities on the continuous model.

Hyperbolic-type spaces of this form have been studied extensively by Li \cite{LiJFA,LiMa,Li2007}, in connection with centered and uncentered maximal functions, heat kernel estimates, and Riesz transforms. His approach relies on the natural Riemannian structure of $S$, which is available when the base space $N$ is a smooth Riemannian manifold. In contrast, our viewpoint is purely metric and does not require any smooth or algebraic structure. In \cite{HQLi4} and in \cite{BBa}, the
authors make use of global Poincaré inequalities on doubling manifolds
to obtain boundedness results for Riesz transforms of Schr\"odinger
operators. In this direction, Theorem \ref{main:theorem} lays the
ground for future work on the subject, in the nondoubling setting.

The paper is organised as follows.
In Section \ref{sec:SW} we establish global Poincaré inequalities for
quasi-flow measures on spiderwebs rooted at infinity whose standard
geodesics have uniformly bounded horizontal components.
Section~\ref{sec:approx} is devoted to the construction and geometric analysis of the
graph $\mathcal{SW}_N$, including the proof of Gromov hyperbolicity and the
quasi-isometry with the continuous hyperbolic-type space $S$.
In Section~\ref{sec:transf} we prove the transference theorem and derive the global
Poincar\'e inequalities on hyperbolic-type spaces.
\medskip

Throughout the paper, the symbol $C$ denotes a positive constant whose
value may change from line to line. Moreover, given two non-negative
quantities $A$ and $B$, we write
\[
A \simeq B
\]
if there exist constants $C_1,C_2>0$ such that
\[
A\leq C_1 B
\qquad\text{and}\qquad
B\leq C_2 A.
\]
Dependence of constants and symbols on relevant parameters is sometimes indicated by
subscripts.
	\section{Poincaré inequalities on Spiderwebs}\label{sec:SW}
    \subsection{Setting and preliminary notions}
In this section we prove the validity of a {\it global}
Poincaré inequality on {\it spiderwebs} equipped with
{\it quasi-flow} measures. We begin by fixing some
terminology and recalling some properties that will be
used throughout the proof.

Let $(X,d_{X}, \rho)$ be a metric measure space. Recall that $\rho$ is locally doubling if for every $R>0$ there is a $C_R>0$ such that
    \begin{align}\label{Locally doubling property on S}
		\rho(B_{X}(x,2r)) \leq C_R \rho(B_{X}(x,r)),\quad \forall x \in X, \forall r\in (0,R].
	\end{align}

A graph is a pair $\mathcal G=(V,E)$, where $V$ is a
countable set of vertices and $E$ is a set of unoriented
edges joining pairs of vertices. We assume that $\mathcal{G}$ is connected and denote by
$d_{\mathcal G}$ the graph distance on $V$ and by
$B_{\mathcal{G}}(z,r)$ the closed ball of center $z\in \mathcal{G}$ and radius $r>0$ with respect to $d_{\mathcal{G}}$.
By a slight abuse of notation, we shall write $x\in \mathcal G$ to mean that
$x\in V$.
If two vertices $x,y\in\mathcal{G}$ are joined by an edge, we write $x\sim y$ and call them neighbours.
A path in $\mathcal G$ is a finite sequence of vertices $\gamma=\{x_0,\ldots,x_n\}$ such that $x_i\sim x_{i+1}$  for every $i=0,\ldots,n-1$. The length of $\gamma$ is denoted by $\ell(\gamma)$ and is equal to the number of edges of the path.

For a function $f$ on a graph  $\mathcal{G}$ we define its discrete gradient as
   \begin{align}
       |\nabla f(x)|=\sum_{y\sim x} |f(x)-f(y)| \qquad \forall x \in \mathcal{G}.
   \end{align}
The spiderwebs introduced in \cite[Definition~2.9]{CMS} are obtained
from a tree rooted at a distinguished vertex by adding suitable
{\it horizontal} edges. In this paper we use a natural horospherical
variant, adapted to multiscale structures indexed by $\mathbb Z$.
\begin{definition}\label{def:sw}
A spiderweb rooted at infinity is a connected graph
$\mathcal{SW}$ endowed with a height function
$\mathcal h:\mathcal{SW}\longrightarrow\mathbb Z$
and a predecessor map $p:\mathcal{SW}\longrightarrow \mathcal{SW}$
with the following properties.
\begin{enumerate}
\item For every $x\in \mathcal{SW}$, $x\sim p(x)$ and $\mathcal  h(p(x))=\mathcal h(x)+1.$
The edges of the form $\{x,p(x)\}$ are called {vertical}.

\item Every nonvertical edge joins two distinct vertices of the same
height. Such edges are called {horizontal}.

\item If $x$ and $y$ are joined by a horizontal edge, then either $p(x)=p(y)$
or $p(x)$ and $p(y)$ are joined by a horizontal edge.
\end{enumerate}
\end{definition}
Notice that the subgraph formed by the vertical edges is a forest, i.e., an acyclic graph,
not necessarily a tree.
Spiderwebs were introduced in \cite{CMS} as $(1,C)$-quasi-isometric discretisations of Gromov hyperbolic spaces and exploited to establish $L^p$-bounds for the
Hardy--Littlewood maximal operator. Such a class is needed since trees are not, in general, quasi-isometric to hyperbolic spaces. For instance, a tree cannot be quasi-isometric to the Poincar\'e disc, since the Gromov boundary of a tree is totally disconnected, whereas the boundary of the Poincar\'e disc is connected. See also \cite{BI} for a different discrete approximation of Gromov hyperbolic spaces.
\\ From now on we denote by $\mathcal{SW}$ a spiderweb rooted at infinity and we will simply call it spiderweb.

For every $x\in\mathcal{SW}$, set
\[
p^0(x)=x,
\qquad
p^m(x)=p\bigl(p^{m-1}(x)\bigr)
\quad\text{for every }m\geq1.
\]
and 
\[s(x)=\{y\in\mathcal{SW}:p(y)=x\},
\qquad
s_0(x)=\{x\},
\qquad
s_m(x)=\bigcup_{y\in s_{m-1}(x)}s(y) \quad\text{for every }m\geq1.
\]
Given $x,y\in\mathcal{SW}$, we write $x\leq y$ if
$y=p^m(x)$
for some nonnegative integer $m$.
The defining property immediately implies that, if
$\gamma=\{x_0,\ldots,x_n\}$
is a horizontal path (i.e., a path connecting vertices at the same height), then, for every nonnegative integer $m$, the
sequence
$\{p^m(x_0),\ldots,p^m(x_n)\}$, after possibly deleting consecutive repetitions, is a horizontal path.
We denote this path by $p^m\gamma$ and have
\begin{align}\label{CMS}
\ell(p^m\gamma)\leq\ell(\gamma).
\end{align}Property \eqref{CMS} plays a crucial role in the proof of the existence of {\it standard geodesics}, established in \cite[Prop. 3.2]{CMS} in the case of vertex rooted spiderwebs. We now adapt this definition to our setting.
\begin{definition}\label{def:standard}
Let $x,y \in \mathcal{SW}$ and $h\ge\max\{\mathcal h(x),\mathcal h(y)\}$.
We denote by $\Gamma(x,y;h)$ the family of paths obtained by
\begin{itemize}
\item[(i)] ascending vertically from $x$ to its ancestor $p^{(h-\mathcal h(x))}(x)$ at height $h$;

\item[(ii)] following a horizontal path in $\mathcal{h}^{-1}(h)$ joining ${p^{(h-\mathcal h(x))}(x)}$ and ${p^{(h-\mathcal h(y))}(y)}$;

\item[(iii)] descending vertically from ${p^{(h-\mathcal h(y))}(y)}$ to $y$.
\end{itemize}
A geodesic $\gamma$ joining $x$ and $y$ in $\mathcal{SW}$ is {\it standard} if $\gamma \in \bigcup_{h \ge \max\{\mathcal h(x),\mathcal h(y)\}}  \Gamma(x,y;h).$
\end{definition} 
Notice that if $\gamma\in\Gamma(x,y;h)$ is a geodesic, then its horizontal component is necessarily a shortest horizontal path between its endpoints. Indeed, otherwise replacing it with a shorter horizontal path would produce a path joining $x$ and $y$ whose length is strictly smaller than $\ell(\gamma)$.
Observe that, for fixed $x,y \in \mathcal{SW}$ and $h \ge \max\{\mathcal{h}(x),\mathcal{h}(y)\}$, the set
$\Gamma(x,y;h)$ may be empty. Indeed, the subgraph
induced by the vertices of $\mathcal{SW}$ at height $h$
is not necessarily connected.

In the vertex rooted
setting introduced in \cite{CMS}, the existence of standard geodesics and the uniform
boundedness of their horizontal components under Gromov
hyperbolicity are proved in \cite[Proposition~3.2]{CMS}.
For the spiderweb $\mathcal{SW}_N$ constructed below, the
corresponding properties will be established directly in
Theorem \ref{thm:normal-form}. For background on Gromov hyperbolic spaces  we refer to \cite[Ch. III.H]{BH}.

    We now introduce an important class of measures on a spiderweb. This generalises flow measures introduced on trees in \cite{LSTV21}.
\begin{definition}
A function $\mu : \mathcal{SW} \to \mathbb R_{+}$ is said to be a quasi-flow if there exists a constant
$C \ge 1$ such that for every $n \in \mathbb{N}$ and every $x \in \mathcal{SW}$
\begin{equation}\label{quasi-flow}
  \sum_{y \in s_n(x)} \mu(y) \le C\,\mu(x).
\end{equation}  
\end{definition}
In order to prove the main result of this section, that is a global Poincaré inequality on $\mathcal{SW}$, we need to assume that the measure $\mu$ is a locally doubling quasi-flow on $\mathcal{SW}$. The local doubling condition alone implies that if $x,y \in \mathcal{SW}$ and $x \sim y$, then $\mu(x) \simeq \mu(y)$ uniformly. 

Indeed, if $x\sim y$, then
\begin{align*}
    \mu(y)& \le \mu(B_{\mathcal{SW}}(x,1)) \le C_{1/2}\mu(B_{\mathcal{SW}}(x,1/2))=C_{1/2}\mu(x).
\end{align*} Interchanging the roles of $x$ and $y$ yields the reverse inequality.
In particular, for any fixed positive integer $M$, if $x \in B_{\mathcal{SW}}(y,M)$ then $\mu(x) \simeq_M \mu(y)$ for a constant only depending on $M$. Observe that if $\mathcal{SW}$ supports a locally doubling measure, then it has bounded geometry, i.e., a vertex has a bounded number of neighbours. Indeed,
\begin{align*}
  \#B_{\mathcal{SW}}(x,1)\mu(x)&\simeq \mu(B_{\mathcal{SW}}(x,1)) \le C_{1/2} \mu(B_{\mathcal{SW}}(x,1/2))=C_{1/2}\mu(x),
\end{align*}
so that $\sup_{x \in \mathcal{SW}}\#\{y \in \mathcal{SW} \ : \ y \sim x\} \le C$.

\begin{remark}

Recall that a flow measure on a rooted tree $\mathcal T$ is a measure $\mu$ satisfying $\sum_{y\in s(x)} \mu(y)=\mu(x)$ for every $x\in\mathcal T$. We use the same terminology for a spiderweb.
A straightforward adaptation of the proof of \cite[Prop.~2.8]{LSTV21} shows that every locally doubling flow measure on a spiderweb in which every vertex has at least two successors has exponential volume growth. In particular, such a measure is not globally doubling.
\end{remark}
	 	\subsection{Poincaré inequalities}
	Global Poincaré inequalities for flow measures were established on trees in \cite[Th. 4.1]{LSTV23}. In the present setting, the horizontal connections of a spiderweb introduce additional difficulties that are not present in the tree case.
	\begin{theorem}\label{Proposition Poinc on SW} Let $\mathcal{SW}$ be a  spiderweb such that every pair of vertices can be joined by a standard geodesic whose horizontal component has uniformly bounded length. Let $\mu$ be any locally doubling quasi-flow on $\mathcal{SW}$.
Then, for every $p \in [1,\infty)$, there exists a constant $C_p$ such that for every  $x_c\in \mathcal{SW}$, $r>0$ and function
$f : \mathcal{SW} \to \mathbb{C}$,    
		\begin{align*}
			\sum_{x \in B_\mathcal{SW}(x_c,r)} |f(x) - f_{B_\mathcal{SW}(x_c,r)}|^p \mu(x) \leq C_p r^p	\sum_{x \in B_\mathcal{SW}(x_c,r)} |\nabla f(x)|^p \mu(x).
		\end{align*}
			\end{theorem}
		\begin{proof}
	Assume $p \in [1, \infty)$. In this proof, we denote by 	$B_{r}$ the ball $B_\mathcal{SW}(x_c,r)$, for $r \geq 1$. We can assume without loss of generality that $r$ is integer.  By Jensen's inequality,
		\begin{align*}
				\sum_{x \in B_r} |f(x) - f_{B_r}|^p \mu(x)  &= \sum_{x \in B_r}  \abs{\sum_{y \in B_r} (f(x) - f(y)) \frac{\mu(y)}{\mu(B_r)} }^p \mu(x) \\
				&\leq \sum_{x \in B_r}  \sum_{y \in B_r} \abs{f(x) - f(y)}^p \frac{\mu(y)}{\mu(B_r)}  \mu(x).
		\end{align*} 
        Consider  standard geodesics $\gamma_x$ and $\gamma_y$ whose horizontal component has uniformly bounded length  connecting $x$ and $y$ to $x_c$, respectively. Let $L_0$ be a uniform upper bound for the length of the horizontal components of the chosen standard geodesics.
Denote by $x'$ and $y'$ the vertices of maximal height on $\gamma_x$ and $\gamma_y$ lying on the vertical geodesic $[x_c,x_r]$ where $x_r=p^r(x_c)$. Observe that, by geodesicity, both $\gamma_x$ and $\gamma_y$ are contained in $B_r$.\\
Let $\gamma_x''$ and $\gamma_y''$ be the subpaths of $\gamma_x$ and $\gamma_y$ joining $x$ to $x'$ and $y$ to $y'$, respectively. 
Define 
\[
\gamma'_x = \gamma_x'' \cup [x',x_r], 
\qquad 
\gamma'_y = \gamma_y'' \cup [y',x_r].
\]
Then $\gamma'_x \cup \gamma'_y$ is a path connecting $x$ to $y$ of length at most $4r$. Hence,
\begin{align}\label{trucco}
|f(x)-f(y)|^p 
&\le \left(\sum_{z \in \gamma'_x} |\nabla f(z)| + \sum_{z \in \gamma'_y} |\nabla f(z)| \right)^p \nonumber \\
&\le C_p\, r^{p/p'} \left( \sum_{z \in \gamma'_x} |\nabla f(z)|^p + \sum_{z \in \gamma'_y} |\nabla f(z)|^p \right),
\end{align} with the obvious modification when $p=1$.
By symmetry and \eqref{trucco}, it suffices to estimate
\[
C_p\, r^{p/p'} \sum_{x \in B_r} \sum_{z \in \gamma'_x} |\nabla f(z)|^p \, \mu(x).
\]
The path $\gamma'_x$ decomposes as
\[
\gamma'_x=[x,x'']\cup [x'',x']\cup [x',x_r],
\]
where $x''\in B_r$. The middle segment $[x'',x']$ is horizontal and has length at most $L_0$, by assumption. Moreover, the segments $[x,x'']$ and $[x',x_r]$ are vertical; see Figure~\ref{figure1}. Set $$B_H(x,R)=B_{\mathcal{SW}}(x,R) \cap \{y \in \mathcal{SW} \ : \ \mathcal{h}(y)=\mathcal{h}(x)\}$$
for every $x \in \mathcal{SW}$ and $R>0$. 
Observe that $z \in \gamma'_x$  only if one of the following holds:
\begin{itemize}
\item[(i)] $x \le z \in B_r$,
\item[(ii)] $z \in B_H(x',L_0) \cap B_r$,
\item[(iii)] $z \in [x',x_r]$.
\end{itemize}
Therefore,
\begin{align*}
C_p\, r^{p/p'} &\sum_{x \in B_r} \sum_{z \in \gamma'_x} |\nabla f(z)|^p \mu(x)
\\&\le C_p\, r^{p/p'} \sum_{x \in B_r}
\left(
\sum_{\substack{z \in B_r \\ x \le z}} |\nabla f(z)|^p
+ \sum_{y \in [x',x_r]} 
\sum_{\substack{z \in B_r \\ z \in B_H(y,L_0)}} |\nabla f(z)|^p
\right)\mu(x).
\end{align*} 
For every $y\in [x',x_r]$ and $u \in \mathcal{SW}$ we say that $u \le B_H(y,L_0)$ if there exists $z \in B_H(y,L_0)$ such that $u \le z$.
By Fubini's theorem,
\begin{align*}
\sum_{x \in B_r}
\Bigg(
\sum_{\substack{z \in B_r \\ z \ge x}} |\nabla f(z)|^p
&+ \sum_{y \in [x',x_r]} 
\sum_{\substack{z \in B_r \\ z \in B_H(y,L_0)}} |\nabla f(z)|^p
\Bigg)\mu(x) \\&\le \sum_{z \in B_r} |\nabla f(z)|^p 
\sum_{\substack{x \in B_r \\ x \le z}} \mu(x)
+\sum_{z \in B_r} |\nabla f(z)|^p\sum_{x \in B_r}\sum_{\substack{y \in[x',x_r] \\ y \in B_H(z,L_0)}} \mu(x) \\&\le \sum_{z \in B_r} |\nabla f(z)|^p 
\left(\sum_{B_r \ni x \le z} \mu(x)+ \sum_{y \in B_H(z,L_0)}\sum_{\substack{x \in B_r \\ x\le B_H(y,L_0)}} \mu(x) \right)\\&=\sum_{z \in B_r} |\nabla f(z)|^p 
\left(I_1(r,z)+I_2(r,z)\right).
\end{align*}
Since  $\mu$ is a quasi-flow, and a descendant of a point can remain inside a ball of radius $r$ for at most $2r$ generations, which equals the diameter of the ball, $I_1(r,z)\le Cr \mu(z)$. 
Similarly
$$\sum_{\substack{x \in B_r \\ x \le B_H(y,L_0)}}\mu(x) = \sum_{v \in B_H(y,L_0)}\sum_{\substack{x \in B_r \\ x \le v}} \mu(x) \le Cr\sum_{v \in B_H(y,L_0)} \mu(v).$$ By the local doubling property, there exists a constant
$C_{L_0}\geq 1$ such that for every $y \in \mathcal{SW}$
\begin{equation*}
    C_{L_0}^{-1}\mu(y)
    \leq \mu(v)
    \leq C_{L_0}\mu(y)
    \qquad
   \forall v\in B_H(y,L_0),
\end{equation*}
and $ \sup_{y\in \mathcal{SW}}\#B_H(y,L_0) \le C$ since $\mathcal{SW}$ has bounded geometry. This implies that $I_2(r,z)\le Cr\mu(z)$ and concludes the proof.
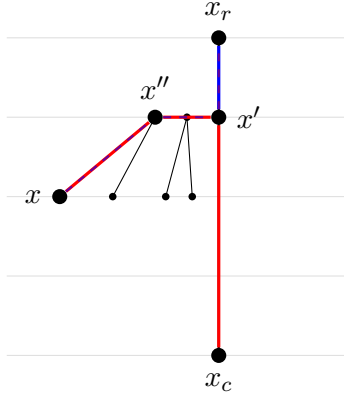
\begin{figure}
\centering
\begin{tikzpicture}[scale=0.7]

% ===== LIVELLI =====
\foreach \y in {0,1.5,3,4.5,6} {
  \draw[gray!25] (-4,\y) -- (2.5,\y);
}

% ===== ASSE CENTRALE =====
\draw[black, thick] (0,0) -- (0,6);

% ===== NODI PRINCIPALI =====
\node[circle,fill,inner sep=2pt,label=below:$x_c$] (xc) at (0,0) {};
\node[circle,fill,inner sep=2pt,label=left:$x$] (x) at (-3,3) {};
\node[circle,fill,inner sep=2pt,label=above:$x^{\prime\prime}$] (xpp) at (-1.2,4.5) {};
\node[circle,fill,inner sep=2pt,label=right:$x^{\prime}$] (xp) at (0,4.5) {};
\node[circle,fill,inner sep=2pt,label=above:$x_r$] (xr) at (0,6) {};

% ===== LIVELLO BASSO (più largo) =====
\foreach \X in {-3,-2,-1,-0.5} {
  \node[circle,fill,inner sep=1pt] at (\X,3) {};
}

% Connessioni tipo albero
\draw (-1,3) -- (-0.6,4.5);
\draw (-2,3) -- (-1.2,4.5);
\draw (-0.5,3) -- (-0.6,4.5);
% ===== LIVELLO INTERMEDIO (più stretto) =====
\foreach \X in {-1.2,-0.6,0} {
  \node[circle,fill,inner sep=1pt] at (\X,4.5) {};
}

% Resta SOLO la connessione verticale verso xr
\draw (0,4.5) -- (0,6);
%\draw (0,4.5) -- (0,7);
% =====  =====
\draw[red, very thick] (x) -- (xpp) -- (xp) -- (xc);

% ===== PROLUNGAMENTO =====
\draw[blue, very thick] (xp) -- (xr);

% ===== GAMMA'_x =====
\draw[blue!60!red, dashed, thick] (x) -- (xpp) -- (xp) -- (xr);

\end{tikzpicture}
\caption{A standard geodesic $\gamma_x$ (in red) connecting $x$ to $x_c$. 
The path first moves vertically to $x^{\prime\prime}$, then horizontally to $x^{\prime}$, and finally vertically to $x_c$. 
The dashed red path represents $\gamma_x''$, connecting $x$ to $x^{\prime}$. 
The extension $\gamma'_x = \gamma_x'' \cup [x^{\prime},x_r]$ is obtained by adding the blue segment.}\label{figure1}
\end{figure}

\end{proof}

\section{Approximation of Hyperbolic-type spaces by spiderwebs}\label{sec:approx}
Given a geodesic doubling metric measure space $(N,d_N,\mu_N)$, we consider the half--space $S = N \times \mathbb R_{+}$, equipped with the hyperbolic-type distance
\begin{equation}\label{eq:hyperbolic-formula}
 d_S\!\bigl((x,y),(x',y')\bigr)
 = C_N\,\mathrm{arcosh}\!\left(
 1+\frac{d_N(x,x')^{2} + (y-y')^{2}}{2yy'}\right) \qquad \forall (x,y),(x',y')\in S,
\end{equation}
where $C_N>0$ is a suitable normalisation depending on $N$.
We first prove that $d_S$ is indeed a distance.
\begin{proposition}\label{distance-rem}
The formula \eqref{eq:hyperbolic-formula} defines a metric on $S$. 
\end{proposition} 
 \begin{proof}
 The only nontrivial fact to verify is the triangle inequality. Consider three distinct points $p_1=(x_1,y_1),p_2=(x_2,y_2),$ and $p_3=(x_3,y_3).$ Then, set $D_{ij}=d_N(x_i,x_j)$ and for any nonnegative $\xi,s,t$ define $\Phi(\xi;s,t)=C_N\mathrm{arcosh}\left(1+\frac{\xi^2+(s-t)^2}{2st}\right),$ so that $d_S(p_i,p_j)=\Phi(D_{ij};y_i,y_j).$ Triangle inequality is equivalent to
\begin{align*}
    \Phi(D_{12};y_1,y_2) \le \Phi(D_{13};y_1,y_3)+\Phi(D_{32};y_3,y_2).
\end{align*} We have that $\Phi$ is increasing in the first variable, so by the triangle inequality on $N$
\begin{align*}
  \Phi(D_{12};y_1,y_2) \le \Phi(D_{13}+D_{32};y_1,y_2).   
\end{align*} Now consider on $\mathbb R \times \mathbb R_{+}$ the points $(a,y_1)$,$(b,y_2)$ and $(c,y_3)$ such that $a\ge c\ge b$ and $a-b=D_{13}+D_{32}\ge 0$, $a-c=D_{13}$, and $c-b=D_{32}$. Then, the triangle inequality on the hyperbolic upper half-plane $\mathbb H^2$ gives
\begin{align*}
  d_S(p_1,p_2) &\le \Phi(D_{13}+D_{32};y_1,y_2)\\&=C_Nd_{\mathbb H^2}\left((a,y_1),(b,y_2)\right) \\&\le  C_Nd_{\mathbb H^2}\left((a,y_1),(c,y_3)\right)+C_Nd_{\mathbb H^2}\left((c,y_3),(b,y_2)\right)\\&=\Phi(D_{13};y_1,y_3)+\Phi(D_{32};y_3,y_2)\\&=d_S(p_1,p_3)+d_S(p_3,p_2).
\end{align*}
     
 \end{proof}

We shall construct a discrete model of $S$ which is a multiscale graph $\mathcal{SW}_N$ built from a family of dyadic
cubes on $N$.   
The resulting graph captures, in a combinatorial framework, the same
multiscale horizontal--vertical geometry that characterizes $(S,d_S)$.

\medskip

The main result of this section shows that $\mathcal{SW}_N$ is 1--quasi-isometric to $S$.
The proof relies on the fact that $\mathcal{SW}_N$ is a spiderweb. Exploiting the spiderweb structure, we establish the existence of standard geodesics whose horizontal component has uniformly bounded length. Moreover, we obtain a quantitative description of the length of such geodesics between arbitrary vertices of the discretisation. This identifies a height at which the predecessors of two vertices lie at uniformly bounded horizontal distance (see Lemma \ref{lem:quantitative-compression}). Combined with the exact formula \eqref{eq:hyperbolic-formula}, this yields matching asymptotics for discrete and continuous distances (see Theorem \ref{thm:1QI}).

Notice that $\mathcal{SW}_N$ is not a priori Gromov hyperbolic. As a further consequence of the geometry of standard geodesics, we prove in Proposition \ref{gromov} that $\mathcal{SW}_N$ is indeed Gromov hyperbolic.

\subsection{Dyadic cubes on $N$}\label{sec:Christ-cubes}
Let $(N,{d_N}, \mu_N)$ be any geodesic doubling metric measure space.
Then, by Christ’s construction \cite{Christ1990Dyadic} and its refinements 
\cite{HytonenKairema2012,HytonenKairema2013},
there exists a system of dyadic cubes adapted to $(N,d_N,\mu_N)$ which we now briefly describe. A dyadic system $\mathscr D=\{\mathscr D_k\}_{k\in \mathbb Z}$ is a family of sets on $N$ such that
$\mathscr D_k=\{Q_\alpha^k\}_\alpha$  consists of the cubes at scale
$k$.
For each $k\in \mathbb Z$, the family $\mathscr D_k$ forms a partition of $N$, 
and the cubes are nested across scales: if $\ell>k$ and 
$Q\in\mathscr D_k$, then there exists a unique $P\in\mathscr D_\ell$ 
such that $Q\subset P$. We denote this ancestor by $Q^{(\ell-k)}$, 
and write $\widehat Q$ when $\ell-k=1$. 
Moreover, there exist constants $\eta>1$ and $c_1,c_2>0$ such that 
for every $Q\in\mathscr D_k$,
\begin{align}\label{inclusions}
B_N(c_Q, c_1\eta^k)\subset Q \subset B_N(c_Q, c_2\eta^k),
\end{align}
where $c_Q$ is a distinguished center of $Q$. In particular,
$\operatorname{diam}(Q)\simeq \eta^k$.
In general, the choice of a center may not be unique, so for each  cube $Q$, we fix once and for all one center $c_Q$ satisfying \eqref{inclusions}.
\subsection{The discrete multiscale graph $\mathcal{SW}_N$}\label{sec:SWN}

    Let $\mathscr D=\{ \mathscr D_k\}_{k \in \mathbb Z}$ be a fixed  dyadic system on $N$.
We now construct a multiscale graph $\mathcal{SW}_N$ which models 
$S = N\times\mathbb R_+$, equipped with \eqref{eq:hyperbolic-formula}.

For each $k\in\mathbb Z$ define
\[
V_k = \{\, z_{Q} : Q\in\mathscr D_k\,\},
\qquad
z_{Q} = (c_Q,\, \eta^k).
\]
Given a set $Q \in \mathscr D_k$ we define the heights
\begin{align}\label{heights}\mathcal h(Q)=k \qquad \text{and} \qquad \mathcal h(z_Q)=\mathcal h(Q).
\end{align}
The full vertex set of $\mathcal{SW}_N$ is $\mathcal V = \bigcup_{k\in\mathbb Z} V_k$. In order to introduce the set of edges of the graph, we need the following notion of enlarged dyadic cube.

\begin{definition}\label{def:1}
For a  dyadic cube $Q \in \mathscr D_k$ we define its enlarged version as
\[
\lambda Q = B_N\!\left(c_Q,\, \lambda\, \eta^{k}\right),
\]
where 
\begin{align}\label{lambda}\lambda >\max\left\{\frac{\eta}{\eta-1} c_2,2c_2+1\right\}
\end{align}is a fixed constant and $c_2$ is as in \eqref{inclusions}.

Two cubes $Q, R \in \mathscr D_k$ are called horizontally adjacent whenever
$\lambda Q \cap \lambda R \neq \emptyset$.
\end{definition}

From now on, whenever the parameter $\lambda$ appears in this section, it is understood to be chosen so that
$\lambda > \max\left\{\frac{\eta}{\eta-1} c_2,\, 2c_2+1\right\}.$
The following elementary lemma collects some consequences of the
dyadic construction and of the doubling property of $N$.
\begin{lemma}\label{removerlap} The following hold.
\begin{itemize}
\item[$i)$] For every $Q \in \mathscr D$\begin{align}\label{lambdaQ}
    \lambda Q \subset \lambda \widehat Q. 
\end{align} 
\item[$ii)$] There exists $C>0$ such that for every $k \in \mathbb Z$ and $R\in\mathscr D_k$  \[\#\{Q\in \mathscr D_{k-1}: Q\subset R\} \le C.\]
\item[$iii)$] There exists $C_\lambda>0$ such that for every $k \in \mathbb Z$ and $x \in N$, $$\#\{Q\in \mathscr D_k : x\in \lambda Q\}\le C_{\lambda}.$$
\end{itemize}
\end{lemma}
\begin{proof} 
Let $Q \in \mathscr D_k$ for some $k \in\mathbb Z$. Observe that $d_N(c_Q,c_{\widehat Q}) \le c_2 \eta^{k+1}$, so that if $x \in \lambda Q$
   \begin{align*} d_N(c_{\widehat Q},x) \le d_N(c_{\widehat Q},c_Q)+d_N(c_{Q},x) \le c_2 \eta^{k+1}+\lambda \eta^k < \lambda\eta^{k+1},
\end{align*} where we have used that $\lambda>\frac{\eta}{\eta-1}c_2.$ Hence $x \in \lambda \widehat Q$. This proves \eqref{lambdaQ}.

For every $k \in \mathbb Z$ and $R \in \mathscr D_k$ set
\[
\mathcal N (R)=\#\{Q\in \mathscr D_{k-1}: Q\subset R\}.\]
Using \eqref{inclusions}, for each child $Q\subset R$ we have $d_N(c_Q,c_{R}) \le c_2 \eta^k,$
and hence $R \subset B_N(c_Q,2c_2 \eta^{k}).$
By \eqref{inclusions} and the doubling condition, this implies that
\[
\mu_N(Q)\ge \mu_N(B_N(c_Q,c_1 \eta^{k-1})) 
\ge C \mu_N(R),
\]
for some constant $C>0$. 
Since the children of $R$ are disjoint and their union is $R$,
\[
\mu_N(R)= \sum_{Q\subset R} \mu_N(Q)
\ge C \mathcal N (R)\,\mu_N(R),
\]
and therefore $\mathcal N (R)\le C$ uniformly in $R$.
The bounded overlap of the enlarged cubes $\{\lambda Q : Q \in \mathscr D_k\}$ follows again from the doubling property of $(N,d_N, \mu_N)$. Indeed, fix $x\in N$.
If $x\in \lambda Q$ for some $Q\in \mathscr D_k$, then by \eqref{inclusions} $$B_N(c_Q,c_1 \eta^k)\subset Q\subset B_N(x,(c_2+\lambda)\eta^k)\subset B_N(c_Q,(c_2+2\lambda)\eta^k).$$   Moreover, by the doubling property
\begin{align*}
    \mu_N(B_N(x, (c_2+\lambda)\eta^k)) \le C_{\lambda} \mu_N(Q).
\end{align*}
Hence
\begin{align*}
\#\{Q\in \mathscr D_k : x\in \lambda Q\}\; C_{\lambda}^{-1}\mu_N(B_N(x,(c_2+\lambda)\eta^k))
   \;&\le\; \sum_{Q\in \mathscr D_k : x\in \lambda Q}\mu_N(Q)
   \;\\ &\le\; \mu_N\!\left(B_N(x,(c_2+\lambda)\,\eta^k)\right).
\end{align*}
We conclude that $$\#\{Q\in \mathscr D_k : x\in \lambda Q\}\; \le {C_{\lambda}},$$ and the constant appearing in the right-hand side does not depend on $x$ and $k$.
\end{proof}

We can finally define the edges of $\mathcal{SW}_N$. 
We say that two distinct vertices are neighbours and we write
$z_{Q} \sim z_{R}$ if and only if exactly one of the following occurs:
\begin{itemize}
    \item[$i)$]  $R=\widehat Q$ or  $Q=\widehat R$, (vertical connection); 
    \item[$ii)$] $\mathcal h(Q)=\mathcal h(R) $, $Q \ne R $, and  $\lambda Q\cap\lambda R\neq\emptyset$  \text{(horizontal connection).} 
\end{itemize}
Notice that $i)$ provides a predecessor map $p(z_Q)={z_{\widehat Q}}$, whereas $ii)$ introduces horizontal connections.

We shall need the following observation. For every
$k\in\mathbb Z$, let $\mathcal G_k$ denote the subgraph of
$\mathcal{SW}_N$ induced by the vertices $V_k$, that is, the
$k$-th layer of $\mathcal{SW}_N$ endowed only with the horizontal edges. Then $\mathcal G_k$ is connected. This
follows from the fact that $N$ is geodesic. Indeed, if $z_Q,z_R\in \mathcal{G}_k$, a geodesic joining $c_Q$ and $c_R$ intersects a chain of dyadic cubes at level $k$ whose consecutive elements are horizontally adjacent, yielding a path in $\mathcal{G}_k$. 
It follows that $\mathcal{SW}_N$ is connected.
\begin{remark}\label{rmk:sw}
Let $Q_1,Q_2\in \mathscr D_k$ and assume that $\lambda Q_1 \cap \lambda Q_2 \neq \emptyset.
$
Let $\widehat Q_1, \widehat Q_2\in \mathscr D_{k+1}$ be their parents.  
Then, \eqref{lambdaQ} implies that $\lambda Q_i\subset \lambda \widehat Q_i$ for $i=1,2$ and thus $\lambda \widehat Q_1 \cap \lambda \widehat Q_2 \ne \emptyset$.  

Hence, either $\widehat Q_1=\widehat Q_2$ and $z_{\widehat Q_1} = z_{\widehat Q_2}$ or $\widehat Q_1\ne\widehat Q_2$ and $z_{\widehat Q_1} \sim z_{\widehat Q_2}$.
This means that horizontal adjacency at height $k$ is either inherited by the parents at
height $k+1$ or the parents coincide. Consequently, $\mathcal{SW}_N$ defined with height function $\mathcal{h}(z_Q)=\mathcal{h}(Q)$ and $p(z_{Q})=z_{\widehat Q}$ is a spiderweb. 
\end{remark}

We highlight some geometric properties of $\mathcal{SW}_N$.

\begin{enumerate}

\item The correspondence
\[
z_Q\longleftrightarrow (c_Q,\eta^{\mathcal h(Q)})
\]
identifies each vertex with a point of $S$ whose vertical coordinate is determined by the height of the cube. 
\item The parent relation $Q\mapsto\widehat Q$ reflects the scaling structure of $S$: passing from $Q$ to $\widehat Q$ increases the vertical coordinate from $\eta^k$ to $\eta^{k+1}$ while simultaneously replacing a region of diameter comparable to $\eta^k$ by one of diameter  approximately $\eta^{k+1}$.

\item The set of vertices provides a uniformly separated and uniformly dense discretisation of $S$. Indeed, given two distinct vertices $z_{Q_1}$ and $z_{Q_2}$, assume that $\mathcal h(z_{Q_1})=k_1$ and $\mathcal h(z_{Q_2})=k_2$ and recall that
\[
d_S\bigl((c_{Q_1},\eta^{k_1}), (c_{Q_2},\eta^{k_2})\bigr)
=
C_N\,\mathrm{arcosh}\!\left(
1+
\frac{d_N(c_{Q_1},c_{Q_2})^2 + (\eta^{k_1}-\eta^{k_2})^2}{2\,\eta^{k_1+k_2}}
\right).
\]
If $k_1=k_2=k$, then $Q_1\cap Q_2=\emptyset$, and  \eqref{inclusions} gives
$d_N(c_{Q_1},c_{Q_2}) \ge c_1 \eta^k$, hence
\[
d_S\bigl((c_{Q_1},\eta^{k}), (c_{Q_2},\eta^{k})\bigr) \ge C_N\,\mathrm{arcosh}(1+c_1^2) > 0.
\]
If instead $k_1>k_2$, then
\begin{align*}
d_S\bigl((c_{Q_1},\eta^{k_1}), (c_{Q_2},\eta^{k_2})\bigr)&\ge C_N\,\mathrm{arcosh}\!\left(1+\frac{(\eta^{k_1}-\eta^{k_2})^2}{2\,\eta^{k_1+k_2}}\right)
\\&\ge C_N\,\mathrm{arcosh}(1+\tfrac{(\eta-1)^2}{2\eta}).
\end{align*}
Thus distinct vertices are uniformly separated.

Conversely, given any $(x,\eta^k)\in S$, choose $Q\in\mathscr D_k$ with $x\in Q$. Then $d_N(x,c_Q)\le c_2\eta^k$, and therefore
\[
d_S\bigl((x,\eta^k),(c_Q,\eta^k)\bigr)
\le C_N\,\mathrm{arcosh}(1+c_2^2).
\]
If $(x,a) \in S$ for some $a \in (\eta^k,\eta^{k+1}), $ then \begin{align*}d_S\left((x,a),(x,\eta^{k})\right)&=C_N\mathrm{arcosh}\left(1+\frac{(\eta^k-a)^2}{2a\eta^k}\right)\le C_N\mathrm{arcosh}\left(\frac{\eta^2+1}{2}\right),
\end{align*}

so that by triangle inequality $$d_S\bigl((x,a),(c_Q,\eta^k)\bigr)
\le C_N\,\left[\mathrm{arcosh}(1+c_2^2)+\mathrm{arcosh}\left(\frac{\eta^2+1}{2}\right)\right].$$

\item By Lemma \ref{removerlap} $ii)$ and $iii)$, each vertex in $\mathcal{SW}_N$ has uniformly bounded degree, with a bound depending only on the dyadic structure of $N$. 
\end{enumerate}
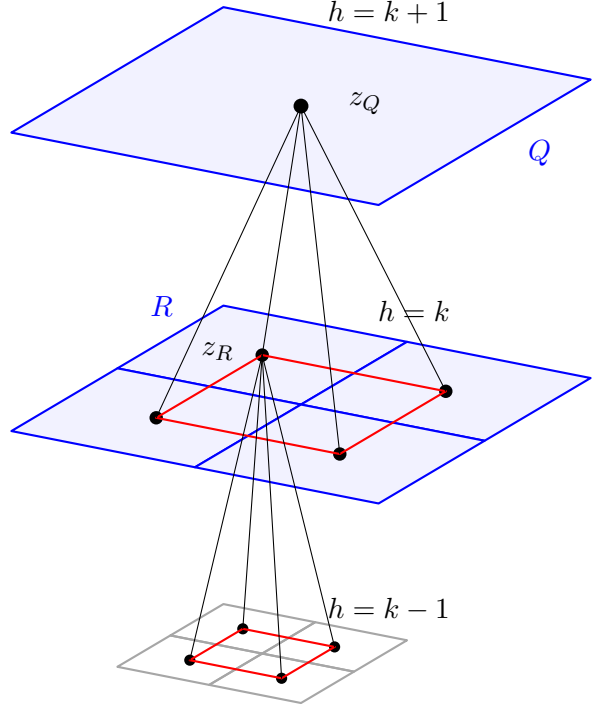
\begin{figure}[ht!]
\tdplotsetmaincoords{70}{120}

\begin{tikzpicture}[
tdplot_main_coords,
scale=1.4,
line join=round,
line cap=round
]

\def\zkm{0}
\def\zk{3}
\def\zkp{6}

%%%%%%%%%%%%%%%%%%%%%%%%%%%%%%%%%%%%%%%%%%%%%%%%%%
% LEVEL k+1
%%%%%%%%%%%%%%%%%%%%%%%%%%%%%%%%%%%%%%%%%%%%%%%%%%

\draw[blue,thick,fill=blue!5]
(-2,-2,\zkp)--(2,-2,\zkp)--(2,2,\zkp)--(-2,2,\zkp)--cycle;

\fill (0,0,\zkp) circle (2pt);

%%%%%%%%%%%%%%%%%%%%%%%%%%%%%%%%%%%%%%%%%%%%%%%%%%
% LEVEL k
%%%%%%%%%%%%%%%%%%%%%%%%%%%%%%%%%%%%%%%%%%%%%%%%%%

\foreach \x in {-2,0}{
  \foreach \y in {-2,0}{
    \draw[blue,thick,fill=blue!5]
    (\x,\y,\zk)--(\x+2,\y,\zk)--(\x+2,\y+2,\zk)--(\x,\y+2,\zk)--cycle;
  }
}

% centers
\fill (-1,-1,\zk) circle (1.8pt);
\fill ( 1,-1,\zk) circle (1.8pt);
\fill (-1, 1,\zk) circle (1.8pt);
\fill ( 1, 1,\zk) circle (1.8pt);

% horizontal spiderweb edges
\draw[red,thick] (-1,-1,\zk)--(1,-1,\zk);
\draw[red,thick] (-1, 1,\zk)--(1, 1,\zk);
\draw[red,thick] (-1,-1,\zk)--(-1,1,\zk);
\draw[red,thick] ( 1,-1,\zk)--( 1,1,\zk);

%%%%%%%%%%%%%%%%%%%%%%%%%%%%%%%%%%%%%%%%%%%%%%%%%%
% LEVEL k-1
% children of the cube centered at (-1,-1,k)
%%%%%%%%%%%%%%%%%%%%%%%%%%%%%%%%%%%%%%%%%%%%%%%%%%

\foreach \x in {-2,-1}{
  \foreach \y in {-2,-1}{
    \draw[gray!70,thick]
    (\x,\y,\zkm)--(\x+1,\y,\zkm)--(\x+1,\y+1,\zkm)--(\x,\y+1,\zkm)--cycle;
  }
}

% centers
\fill (-1.5,-1.5,\zkm) circle (1.5pt);
\fill (-0.5,-1.5,\zkm) circle (1.5pt);
\fill (-1.5,-0.5,\zkm) circle (1.5pt);
\fill (-0.5,-0.5,\zkm) circle (1.5pt);

% horizontal edges
\draw[red,thick] (-1.5,-1.5,\zkm)--(-0.5,-1.5,\zkm);
\draw[red,thick] (-1.5,-0.5,\zkm)--(-0.5,-0.5,\zkm);
\draw[red,thick] (-1.5,-1.5,\zkm)--(-1.5,-0.5,\zkm);
\draw[red,thick] (-0.5,-1.5,\zkm)--(-0.5,-0.5,\zkm);

%%%%%%%%%%%%%%%%%%%%%%%%%%%%%%%%%%%%%%%%%%%%%%%%%%
% vertical edges
%%%%%%%%%%%%%%%%%%%%%%%%%%%%%%%%%%%%%%%%%%%%%%%%%%

% k -> k+1
\foreach \x/\y in {-1/-1,1/-1,-1/1,1/1}{
   \draw[black] (\x,\y,\zk)--(0,0,\zkp);
}

% k-1 -> selected cube at level k
\foreach \x/\y in {
-1.5/-1.5,
-0.5/-1.5,
-1.5/-0.5,
-0.5/-0.5}{
   \draw[black] (\x,\y,\zkm)--(-1,-1,\zk);
}

%%%%%%%%%%%%%%%%%%%%%%%%%%%%%%%%%%%%%%%%%%%%%%%%%%
% labels
%%%%%%%%%%%%%%%%%%%%%%%%%%%%%%%%%%%%%%%%%%%%%%%%%%

\node[left] at (-3,0,\zkm) {$h=k-1$};
\node[left] at (-3,0,\zk) {$h=k$};
\node[left] at (-3,0,\zkp) {$h=k+1$};
\node[left] at (-0.5,0.7,\zkp) {$z_Q$};
\node[blue] at (0,2.6,\zkp) {$Q$};
\node[left] at (-1,-1.2,\zk) {$z_R$};
\node[blue] at (-1.7,-2.5,\zk) {$R$};
%%%%%%%%%%%%%%%%%%%%%%%%%%%%%%%%%%%%%%%%%%%%%%%%%%
% continuation
%%%%%%%%%%%%%%%%%%%%%%%%%%%%%%%%%%%%%%%%%%%%%%%%%%

%\node at (2.7,-1,\zk) {$\cdots$};
%\node at (-2.8,1,\zk) {$\cdots$};
\end{tikzpicture}
\caption{A local portion of the spiderweb $\mathcal{SW}_{\mathbb R^2}$
associated with the dyadic decomposition of $\mathbb R^2$. Vertices correspond
to dyadic squares, black edges encode the parent--child relation, and red edges
represent horizontal connections between neighbouring squares.}
\end{figure}
\subsection{Standard form of geodesics in $\mathcal{SW}_N$}

We now prove that every pair of vertices in $\mathcal{SW}_N$ can be joined by a special standard geodesic: a geodesic that first ascends vertically, then traverses a {\it uniformly bounded} horizontal segment at an appropriate height, and finally descends vertically. This is the analogue of the standard form of geodesics on hyperbolic rooted spiderwebs established in \cite{CMS}.
Let $\gamma = \{z_0,z_1,\dots,z_m\}$ be any path in $\mathcal{SW}_N$.
The  \emph{height of the path} $\gamma$ is defined as
\begin{align}\label{heightofapath}
h(\gamma) = \max_{0\le j\le m}\, \mathcal h(z_j).
\end{align}
\begin{lemma}
\label{lem:max-height-ancestors}
Given any two vertices $z_{Q}$ and $z_{R}$ in $\mathcal{SW}_N$ and any path $\gamma$ joining
them, if $h=h(\gamma)$, then there exist unique ancestors 
$Q_h\in\mathscr{D}_h$ and $R_h\in\mathscr{D}_h$ such that
\[
Q\subset Q_h,\qquad R\subset R_h.
\]
\end{lemma}

\begin{proof}
Every vertical step in $\mathcal{SW}_N$ corresponds to moving from a cube to its parent,
which is unique by the property of the dyadic family. Thus repeated application of parents yields a unique ancestor at height $h$.
\end{proof}
The central idea underlying the construction of standard geodesics is the following: as one ascends, the cubes increase their diameter at a controlled rate. After finitely many steps (which we will quantify), the ancestors of two cubes become “large enough” to lie within a uniformly bounded horizontal graph distance. The argument relies only on the doubling property of $N$, the nesting of dyadic cubes, and the
uniform ball inclusion property \eqref{inclusions}.

\begin{definition} For every $k \in \mathbb Z$ and $x,y \in \mathcal{G}_k$ we set $$d_H(x,y)=d_{\mathcal{G}_k}(x,y),$$ and we refer to $d_H$ as ``horizontal distance'' since  it is the distance on a fixed layer. 
\end{definition}

We shall make use of the following fundamental property.

\begin{proposition}\label{quasi-iso-hor}
There exist positive constants $K_1,K_2$ such that, for every
$k\in\mathbb Z$ and every $Q,R\in\mathscr D_k$
\begin{align}\label{bilipsc}
K_1\, d_N(c_Q,c_{R}) \;\le\; \eta^k\, d_H(z_{Q},z_{R}) \;\le\; K_2\, d_N(c_Q,c_{R}).
\end{align}
In other words, $d_H$ and $d_N/\eta^k$ are bi-Lipschitz equivalent on $\mathcal{G}_k$.
\end{proposition}

\begin{proof}

Let $z_{Q} = z_{S_0} \sim z_{S_1} \sim \dots \sim z_{S_\ell} = z_{R}$
be a shortest path in $\mathcal{G}_k$. Then $d_H(z_{Q},z_{R}) = \ell$. By the triangle inequality, the fact that $\lambda S_i \cap \lambda S_{i+1} \ne \emptyset$ and that $(N,d_N)$ is geodesic, we get
\[
d_N(c_Q,c_{R}) \le \sum_{i=0}^{\ell-1} d_N(c_{S_i},c_{S_{i+1}}) \le 2\lambda \eta^k \ell,
\] with $\lambda$ as in \eqref{lambda}.
Thus
\[
\eta^k d_H(z_{Q},z_{R}) = \eta^k \ell \ge \frac{1}{2\lambda} d_N(c_Q,c_{R}),
\]
giving the lower bound in \eqref{bilipsc} with $K_1 = 1/(2\lambda)$.

\medskip

For the reverse inequality, let $x = c_Q$, $y = c_{R}$ for some $Q,R \in \mathscr D_k$ and set $D = d_N(x,y)$. 
\\ If $D<2c_1\eta^k$, let $\gamma$ be a geodesic joining $c_Q$ to $c_R$ and let
$p$ be its midpoint. Then
\[
d_N(p,c_Q)=d_N(p,c_R)=\frac{D}{2}<c_1\eta^k,
\]
so that
\[
p\in B_N(c_Q,c_1\eta^k)\cap B_N(c_R,c_1\eta^k).
\]
Since
\[
B_N(c_Q,c_1\eta^k)\subset Q,
\qquad
B_N(c_R,c_1\eta^k)\subset R,
\]
we obtain $Q\cap R\neq\emptyset$. As dyadic cubes at the same height form a
partition of $N$, it follows that $Q=R$, and therefore $z_Q=z_R$, $c_Q=c_R$, and $D=0$. Hence we may assume $D \ge 2 c_1 \eta^k$. \\
Let $\gamma : [0,D] \to N$ be a unit-speed geodesic from $x$ to $y$. Fix
\[
\delta \in (2c_2 \eta^k, \lambda \eta^k).
\]
 Set
\[
t_i = i \delta, \quad i = 0,\dots, m-1, \quad t_m = D, \quad m = \lfloor D/\delta \rfloor + 1.
\]
If $m \le 2$,  then $2c_1 \eta^k\le D<2\delta< 2\lambda \eta^k$ so that by geodesicity of $N$, the upper bound implies that $\lambda Q \cap \lambda R \ne \emptyset$ and thus $z_{Q} \sim z_{R}$. It follows that $d_H(z_{Q},z_{R})=1$ and  $$\frac{d_N(c_Q,c_R)}{2\lambda}\le \eta^k d_H(z_{Q},z_{R}) \le \frac{d_N(c_Q,c_R)}{2c_1},$$
which gives the upper bound in \eqref{bilipsc} with $K_2\ge \frac{1}{2c_1}$.

If $m>2$, then we define $x_i = \gamma(t_i)$, so that
\[
\delta=d_N(x_i,x_{i+1})  \quad \text{for all } i\le m-2 \qquad \text{and}\quad d_N(x_{m-1},x_m) \le \delta.
\]

For each $i\in \{0,...,m\}$, choose $Q_i \in \mathscr{D}_k$ such that $x_i \in Q_i$. Then by \eqref{inclusions}:
\[
d_N(c_{Q_i},x_i) \le c_2 \eta^k.
\]

Observe that $Q_i \neq Q_j$ for all distinct  $i,j \in \{0,...,m-1\}$. Indeed, otherwise:
\[
d_N(x_i,x_j) \le d_N(x_i,c_{Q_i}) + d_N(c_{Q_j},x_j) \le 2 c_2 \eta^k < |i-j| \delta = d_N(x_i,x_j),
\]
a contradiction. In particular, $Q_0 = Q$ and $Q_m = R$.

Next, for each $i$,
\begin{align*}
d_N(c_{Q_i},c_{Q_{i+1}}) &\le d_N(c_{Q_i},x_i) + d_N(x_i,x_{i+1}) + d_N(x_{i+1},c_{Q_{i+1}}) \\&\le 2c_2 \eta^k + \delta \\&< 2 \lambda \eta^k.
\end{align*}
By the adjacency condition of dyadic cubes and geodesicity of $d_N$, this implies $z_{Q_i} \sim z_{Q_{i+1}}$ in $\mathcal{G}_k$ for every $i=0,...,m-2$ (we cannot a priori exclude that $Q_{m-1}= Q_m$).

Hence the sequence $z_{Q_0},\dots,z_{Q_{m-1}}$ defines a horizontal path from $z_{Q}$ to $z_{Q_{m-1}}$ of length at most $m\le D/\delta +1$ and either $z_{R}\sim z_{Q_{m-1}}$ or $z_{R}=z_{Q_{m-1}}$. Using $\delta> 2c_2 \eta^k$, we get
\[
d_H(z_{Q},z_{R}) \le m \le \frac{D}{\eta^k} \cdot \frac{1}{2c_2} + 1\le \frac{D}{\eta^k2c_2} + \frac{D}{\eta^k2c_1}.
\]
Setting $K_2 = \frac{1}{2c_2}+ \frac{1}{2c_1}$ gives the desired upper bound in \eqref{bilipsc}.\\
Combining the two estimates, we obtain the bi-Lipschitz equivalence.
\end{proof}

One of the main tools to derive the standard form of geodesics is the fact that the horizontal distances shrink as we pass 
to higher predecessors in the dyadic hierarchy.  
This phenomenon is made precise in the following lemma.

\begin{lemma}
\label{lem:vertical-compression}
Let $Q,R\in\mathscr{D}_k$ and let $Q^{(1)},R^{(1)}\in\mathscr{D}_{k+1}$ be their
unique parents. Then
\begin{equation}
\label{eq:vertical-compression}
\frac{d_N\!\bigl(c_{Q^{(1)}},c_{R^{(1)}}\bigr)}{\eta^{\,k+1}}
\;\le\;
\frac1\eta\,\frac{d_N\!\bigl(c_Q,c_{R}\bigr)}{\eta^{\,k}}
\;+\; C_0,
\end{equation}
where $C_0 = 2c_2$.
Consequently, for every $m\ge 1$,
\begin{equation}
\label{eq:vertical-compression-iterated}
\frac{d_N\!\bigl(c_{Q^{(m)}},c_{R^{(m)}}\bigr)}{\eta^{\,k+m}}
\;\le\;
\eta^{-m}\,\frac{d_N\!\bigl(c_Q,c_{R}\bigr)}{\eta^{\,k}}
\;+\; \frac{\eta}{\eta-1}C_0.
\end{equation}
In particular,   \begin{align}\label{stella}
d_H({z_{Q^{(m})}},z_{R^{(m)}}) \le \frac{K_2}{K_1}\frac{d_H(z_{Q},z_{R})}{\eta^m}+K_2\frac{\eta}{\eta-1}C_0,
\end{align} with $K_1$ and $K_2$ as in Proposition \ref{quasi-iso-hor}.
\end{lemma}

\begin{proof}
By \eqref{inclusions},
\[
d_N\!\bigl(c_Q,c_{Q^{(1)}}\bigr)\;\le\;c_2\,\eta^{\,k+1}, 
\qquad
d_N\!\bigl(c_{R},c_{R^{(1)}}\bigr)\;\le\;c_2\,\eta^{\,k+1}.
\]
Hence by the triangle inequality,
\[
d_N\!\bigl(c_{Q^{(1)}},c_{R^{(1)}}\bigr)
\;\le\;
d_N\!\bigl(c_Q,c_{R}\bigr)
\,+\,
2c_2\,\eta^{\,k+1}.
\]
Dividing both sides by $\eta^{k+1}$,
\[
\frac{d_N(c_{Q^{(1)}},c_{R^{(1)}})}{\eta^{\,k+1}}
\;\le\;
\frac{1}{\eta}\frac{d_N(c_Q,c_{R})}{\eta^{\,k}}
\,+\,
2c_2.
\]
which proves \eqref{eq:vertical-compression} with $C_0=2c_2$. Iterating \eqref{eq:vertical-compression} gives
\[
\frac{d_N(c_{Q^{(m)}},c_{R^{(m)}})}{\eta^{k+m}}
\;\le\;
\eta^{-m}\frac{d_N(c_Q,c_{R})}{\eta^{k}}
\,+\,
C_0\,(1+\eta^{-1}+\eta^{-2}+\cdots),
\]
which yields \eqref{eq:vertical-compression-iterated}.
\end{proof}

Recall that for
vertices $z_Q,z_R\in \mathcal{SW}_N$ and
$h\ge \max\{\mathcal h(Q),\mathcal h(R)\}$, we denote by $\Gamma(z_Q,z_R;h)$
the family of paths obtained by
\begin{itemize}
\item[(i)] ascending vertically from $z_Q$ to its ancestor
$z_{Q^{(h-\mathcal h(Q))}}$ at height $h$;

\item[(ii)] following a horizontal path in $\mathcal{G}_h$ joining
$z_{Q^{(h-\mathcal h(Q))}}$ and
$z_{R^{(h-\mathcal h(R))}}$;

\item[(iii)] descending vertically from
$z_{R^{(h-\mathcal h(R))}}$ to $z_R$.
\end{itemize}

We now have all the ingredients to prove the existence of standard geodesics.

\begin{theorem}\label{thm:normal-form}
Let $z_{Q}$ and $z_{R}$ be vertices of $\mathcal{SW}_N$. 
Then for every geodesic $\gamma$ joining them there is a standard geodesic
$\gamma' \in \Gamma(z_Q,z_R;h)$ such that:
\begin{enumerate}
\item $h = h(\gamma)$ is the maximal height reached by the geodesic $\gamma$,
\item  if $Q_h,R_h \in \mathscr{D}_h$ are the ancestors of $Q,R$ at height $h$, the horizontal segment of $\gamma'$ satisfies
\begin{align}\label{hor:length}
d_H(z_{Q_h},z_{R_h}) \le K_N,
\end{align}
for a constant $K_N$ depending only on $N$.
\end{enumerate}
\end{theorem}

\begin{proof}
We distinguish two cases. If $Q \cap R \ne \emptyset$, then either $Q \subset R$ or $R \subset Q$. In this case the geodesic connecting $z_{Q}$ and $z_{R}$ is unique and it coincides with the vertical geodesic connecting the two vertices (this corresponds to a degenerate case of the statement, in which the horizontal segment is trivial and there is no descending part).

Now assume that $Q \cap R = \emptyset$. Set $k=\mathcal h(Q)$ and $\ell= \mathcal h(R)$.
Let $\gamma$ be a geodesic joining $z_{Q}$ and $z_{R}$, and let $h=h(\gamma)$ be its height defined as in \eqref{heightofapath}.  
By Lemma~\ref{lem:max-height-ancestors}, $\gamma$ determines unique ancestors 
$Q_h,R_h \in \mathscr{D}_h$ of $Q$ and $R$, respectively.

\medskip

\noindent

By Lemma~\ref{lem:vertical-compression}, for every $m\ge 1$ we have
\begin{equation}\label{eq:metric-compression}
\frac{d_N\left(c_{Q_{h}^{(m)}},c_{R_{h}^{(m)}}\right)}{\eta^{h+m}}
\;\le\;
\eta^{-m}\,\frac{d_N(c_{Q_{h}},c_{R_{h}})}{\eta^h}
\;+\; C.
\end{equation}
\noindent
Set
\[
E_h = \frac{d_N(c_{Q_{h}},c_{R_{h}})}{\eta^h}, \qquad E_{h+m} = \frac{d_N(c_{Q_{h}^{(m)}},c_{R_{h}^{(m)}})}{\eta^{h+m}}\quad \forall m\ge1.
\]
We claim that $E_h$ must be bounded by a constant depending only on $N$.

Indeed, consider the path obtained by modifying $\gamma$ as follows:
\begin{itemize}
\item[i')] from $z_{Q}$ go vertically up to height $h+m$,
\item[ii')] move horizontally from $z_{Q_{h}^{(m)}}$ to $z_{R_{h}^{(m)}}$,
\item[iii')] descend to $z_{R}$.
\end{itemize}

This new path has $2m+h-k+h-\ell$  vertical edges and  horizontal length bounded above by
\(
\displaystyle K_2E_{h+m}
\) by Proposition \ref{quasi-iso-hor}.
Using \eqref{eq:metric-compression}, we deduce
\[
E_{h+m}\le \eta^{-m} E_h + C.
\]

On the other hand, we can easily estimate the length of $\gamma$ from below. Indeed, by the spiderweb property of horizontal paths, by decomposing the horizontal component of $\gamma$ into maximal horizontal
subpaths, that is, maximal subpaths consisting of consecutive vertices lying at the same height and joined by horizontal edges, and applying \eqref{CMS} to each of them we obtain that the total
horizontal length of $\gamma$ is bounded from below by the length of its
projection onto height $h$, that in turn is at least $K_1 E_h$ by Proposition~\ref{quasi-iso-hor}. The vertical component of $\gamma$ has length at least $h-k+h-\ell$. Hence the difference in length between the modified path and $\gamma$ is at most
\[
2m + K_2(\eta^{-m}E_h + C) - K_1 E_h,
\]
where $K_1$ and $K_2$ are as in Proposition \ref{quasi-iso-hor}. Fix $m \ge \log_\eta (2K_2/K_1)$. Then
\begin{align*}
   2m + K_2(\eta^{-m}E_h + C) - K_1 E_h\le 2m -K_1/2E_h + K_2C,
\end{align*}
and the above quantity is negative if $E_h$ is bigger than a constant depending only on $K_1,K_2$ and $N$.
 This would contradict the geodesicity of $\gamma$  and therefore $E_h$ is uniformly bounded.

\medskip

\noindent

By Proposition~\ref{quasi-iso-hor}, we have
\[
d_H(z_{Q_h},z_{R_h})
\;\le\;
K_2\, \frac{d_N(c_{Q_{h}},c_{R_{h}})}{\eta^h}
\;=\;
K_2\, E_h.
\]
Since $E_h$ is bounded, this yields \eqref{hor:length} for some constant depending on $N$.

\medskip

\noindent

Finally, we prove that $\gamma'$ is a geodesic. We decompose $\gamma$ into maximal horizontal subpaths. For each horizontal subpath $\sigma$, let $\pi_h(\sigma)$ denote the path obtained by replacing every vertex of $\sigma$ with its ancestor at height $h$. By the spiderweb property and \eqref{CMS}, $\pi_h(\sigma)$ is again a horizontal path and its length is not larger than that of $\sigma$. The endpoints of every horizontal subpath are joined in $\gamma$ by vertical segments. Since $\pi_h$ is constant along each vertical segment, projected consecutive horizontal subpaths have matching endpoints. Concatenating these projected horizontal subpaths with the initial and
final vertical segments from $z_Q$ to $z_{Q_h}$ and from $z_R$ to
$z_{R_h}$ gives a path $\gamma'\in\Gamma(z_Q,z_R;h)$ joining $z_Q$ to $z_R$.

The total horizontal length of $\gamma'$ is therefore at most that of $\gamma$,
and the same holds for the vertical components. Hence
$\ell(\gamma')\le \ell(\gamma)$. Since $\gamma$ is a geodesic, equality must
hold, and therefore $\gamma'$ is itself a geodesic.
\end{proof}
We record one further consequence of Theorem~\ref{thm:normal-form}, which will be
used later in the proof of the Gromov hyperbolicity of
$\mathcal{SW}_N$.\\
For $A\subset\mathcal{SW}_N$ and $C>0$, we set
\begin{align}\label{intorno}
    \mathcal N_C(A)
    =
    \{v\in\mathcal{SW}_N:
      d_{\mathcal{SW}_N}(v,A)\leq C\}.
      \end{align}

\begin{lemma}
\label{lem:standard-geodesic-approximation}
There is a constant $C>0$ such that for every pair of vertices $x,y\in\mathcal{SW}_N$ and every geodesic
$\gamma$ joining $x$ to $y$, there exists a standard geodesic
$\gamma'$ joining the same endpoints such that
\[
    \gamma\subset \mathcal N_C(\gamma')
    \qquad\text{and}\qquad
    \gamma'\subset \mathcal N_C(\gamma).
\]
In particular, $\gamma$ and $\gamma'$ are at Hausdorff distance at
most $C$.
\end{lemma}

\begin{proof}
    Let $\gamma$ be a geodesic joining $x$ to $y$ and set $h=h(\gamma)$.
By Theorem~\ref{thm:normal-form}, there exists a standard geodesic
$\gamma'\in\Gamma(x,y;h)$, constructed as in its proof, whose
horizontal component has uniformly bounded length. For any path $\sigma$, denote by $H(\sigma)$ and $V(\sigma)$ the total numbers of horizontal and vertical edges of $\sigma$, respectively. In the construction given in the proof of Theorem \ref{thm:normal-form}, one has $H(\gamma')\le H(\gamma)$ and $V(\gamma')\le V(\gamma)$.
Since both $\gamma$ and $\gamma'$ are geodesics joining the same endpoints, they have the same length. Therefore, the two inequalities above must in fact be equalities:
\[
H(\gamma')=H(\gamma),
\qquad
V(\gamma')=V(\gamma).
\]
Consequently, the total number of horizontal edges of $\gamma$ is uniformly bounded. Moreover,
\[
V(\gamma)= V(\gamma')=
h-\mathcal h(x)+h-\mathcal h(y).
\]
Let $U$ and $D$ denote respectively the numbers of upward and downward vertical edges of $\gamma$, when $\gamma$ is oriented from $x$ to $y$. Since $\gamma$ reaches height $h$, we necessarily have
\[
U\ge h-\mathcal h(x),
\qquad
D\ge h-\mathcal h(y).
\]
Since
\[
U+D
= V(\gamma)=h-\mathcal h(x)+h-\mathcal h(y),
\]
both inequalities must be equalities, thus
$U=h-\mathcal h(x)$ and $D=h-\mathcal h(y)$.
Consequently, no downward vertical edge can occur before $\gamma$ first reaches height $h$, and no upward vertical edge can occur afterwards.

Let $z$ be a vertex of the ascending part of $\gamma$, with $\mathcal h(z)=k$. Projecting to height $k$ all the horizontal edges occurring in the portion of $\gamma$ between $x$ and $z$, and using the spiderweb property, we obtain a horizontal path joining $z$ and 
$p^{k-\mathcal h(x)}(x)$,
whose length is bounded by the total number of horizontal edges of $\gamma$. Hence
$d_{\mathcal{SW}_N}
\left(
z,p^{k-\mathcal h(x)}(x)
\right)
\le C.$
An analogous argument applies to the descending part of $\gamma$ and the vertical component of $\gamma'$ joining $p^{h-\mathcal h(y)}(y)$ to $y$. The portion of $\gamma$ lying at height $h$ is also contained in a uniform neighborhood of $\gamma'$, since the total number of horizontal edges of $\gamma$ is uniformly bounded. We conclude that $\gamma\subset\mathcal N_C(\gamma').$

Conversely, every height between $\mathcal h(x)$ and $h$ is attained by the ascending part of $\gamma$. Thus, for every integer $k$ with $\mathcal h(x)\le k\le h,$
there exists a vertex $z_k$ on the ascending part of $\gamma$ with $\mathcal h(z_k)=k$. By the estimate above, $d_{\mathcal{SW}_N}
\left(
z_k,p^{k-\mathcal h(x)}(x)
\right)
\le C.$
Therefore every vertex of the vertical component of $\gamma'$ issuing from $x$ lies at uniformly bounded distance from $\gamma$. The same argument applies to the vertical component of $\gamma'$ issuing from $y$. Since the horizontal component of $\gamma'$ has uniformly bounded length, we conclude that $\gamma'\subset\mathcal N_C(\gamma).$
Thus $\gamma$ and $\gamma'$ are at uniformly bounded Hausdorff distance.
\end{proof}

We now return to the quantitative analysis of standard geodesics and
compute the distance between arbitrary vertices
$z_Q,z_R\in\mathcal{SW}_N$.
By Theorem~\ref{thm:normal-form}, if $Q \cap R =\emptyset$, any standard
geodesic must rise to a height $h$ large enough that the horizontal
distance between the ancestors $Q_h$ and $R_h$ is uniformly bounded,
and then descend.

In order to compute the distance between two points in $\mathcal{SW}_N$, we need a precise 
estimate of the maximal height of a standard geodesic in terms of the distance in $N$
between its endpoints.   The next lemma provides the first explicit upper bound for the height at which two dyadic chains become horizontally adjacent. This estimate will be the key ingredient in the computation of the maximal height of standard geodesics.
\begin{lemma}
\label{lem:quantitative-compression}
Let $k \in \mathbb Z$ and $Q,R \in \mathscr{D}_k$ be  dyadic cubes at height $k$, and set $D = d_N(c_Q,c_{R}).$ Define $m = \min \bigl\{ m \in \mathbb{N} : \eta^{k+m} \ge  D \bigr\}.$
Then the ancestors $Q^{(m)}, R^{(m)} \in \mathscr{D}_{k+m}$ satisfy
\[
\lambda Q^{(m)} \cap \lambda R^{(m)} \neq \emptyset.
\]
In particular, the corresponding vertices satisfy $d_H\bigl(z_{Q^{(m)}}, z_{R^{(m)}}\bigr) \le 1.$
\end{lemma}

\begin{proof}
By \eqref{inclusions}
\[
d_N(c_{Q^{(m)}},c_Q) \le c_2 \eta^{k+m}, \qquad
d_N(c_{R^{(m)}},c_{R}) \le c_2 \eta^{k+m}.
\]
Hence
\[
d_N(c_{Q^{(m)}},c_{R^{(m)}}) \le D + 2 c_2 \eta^{k+m}.
\]
By definition of $m$, we have $\eta^{k+m} \ge D$, so by \eqref{lambda}
\[
D + 2 c_2 \eta^{k+m} \le (2c_2+1) \eta^{k+m} <\lambda \eta^{k+m}.
\]
This implies that $\lambda Q^{(m)} \cap \lambda R^{(m)} \neq \emptyset,$
and therefore the corresponding vertices in $\mathcal G_{k+m}$ are adjacent, i.e., $d_H\bigl(z_{Q^{(m)}}, z_{R^{(m)}}\bigr) \le 1.$
\end{proof}

We can now provide an estimate of the maximal height $h$ of a standard geodesic.

\begin{proposition}\label{prop:hmax}
Let $z_{Q}$ and $z_{R}$ be distinct vertices in $\mathcal{SW}_N$ and assume that $Q \cap R= \emptyset$ (so that there are no inclusion relations between $Q$ and $R$). Let
$\gamma$ be a geodesic in standard form connecting them.  
Let $h=h(\gamma)$ be its maximal height defined as in \eqref{heightofapath}.  
Then
\[
h \;=\; \log_\eta\bigl( d_N(c_Q,c_{R}) \bigr) \;+\; O(1).
\]
\end{proposition}

\begin{proof}
Let $D = d_N(c_Q,c_{R})$ and let $Q_h,R_h$ be the ancestors of $Q,R$ at height $h$.
Since $\gamma$ is standard, its horizontal segment at height $h$ has uniformly
bounded length, hence by Proposition~\ref{quasi-iso-hor},
\begin{align}\label{eq:upper-small}
\frac{d_N(c_{Q_{h}},c_{R_{h}})}{\eta^h} \;\le\; C.
\end{align}
Using \eqref{inclusions},
\[
d_N(c_Q,c_{R})
\;\le\;
d_N(c_{Q_{h}},c_{R_{h}}) + d_N(c_Q,c_{Q_{h}})+d_N(c_{R},c_{R_{h}})
\le C \eta^h.
\]
Thus
\begin{align}\label{eq:lower-bound-on-h}
h \;\ge\; \log_\eta D - C.
\end{align}

Now we prove the reverse inequality.  
The idea is that, once the two ancestor cubes become horizontally adjacent, any additional ascent only produces unnecessary vertical cost without reducing the horizontal part. 

Indeed, suppose without loss of generality that $k=\mathcal h(z_{Q}) \ge \ell= \mathcal h(z_{R})$.
We shall apply Lemma~\ref{lem:quantitative-compression} at height $k$ to $Q$ and $R^{(k-\ell)} \in \mathscr D_k$. Set $D_k=d_N(c_Q,c_{R^{(k-\ell)}})$. 
By assumption $R^{(k-\ell)} \ne Q$.

Observe that \begin{align*}
D_k&\le d_N(c_Q, c_{R})+d_N(c_{R},c_{R^{(k-\ell)}}) \\&\le d_N(c_Q, c_{R})+c_2\eta^k \\&\le C d_N(c_Q, c_{R})
\\&=CD,
\end{align*}
because we are assuming that $R \not \subset Q$, so $d_N(c_Q,c_{R}) \ge c_1 \eta^k$. Then, by Lemma \ref{lem:quantitative-compression} 
there exists an integer $m$, with
\begin{align}\label{eq:m}
m \;\le\;\log_\eta\!\Bigl(\frac{D_k}{\eta^{k}}\eta\Bigr)+C\le \log_\eta\!\Bigl(\frac{D}{\eta^{k}}\Bigr)+C,
\end{align}
such that the ancestors $Q^{(m)},R^{(m+k-\ell)}$ at height $h_* = k+m$
satisfy
\begin{align}\label{eq:hstar}
d_H\!\bigl( z_{Q^{(m)}},\, z_{R^{(m+k-\ell)}} \bigr)\;\le\;1.
\end{align}

Hence \eqref{eq:hstar}  implies that at height $h_*$ the predecessors are already horizontally adjacent.
But if $h>h^*$, then using height $h$ instead of $h^*$ would require
$2(h-h^*)$ additional vertical edges, while reducing the horizontal
length by at most one. Hence the resulting path would be strictly
longer, contradicting the minimality of $\gamma$.
Therefore $h \;\le\; h_* \;=\; k+m.$
\\ Combining \eqref{eq:m} and \eqref{eq:lower-bound-on-h} gives
\[
h
=
\log_\eta D + O(1),
\] as required.
\end{proof}

\subsection{Comparison with the continuous hyperbolic-type model}\label{sec:QI}

We now show that the discrete space $\mathcal{SW}_N$, equipped with its natural
graph metric, is $1$--quasi--isometric (up to an additive constant)
to the space $S = N \times \mathbb R_{+}$
endowed with the metric \eqref{eq:hyperbolic-formula} with $C_N=\frac{1}{\log\eta }$.
\begin{theorem}
\label{thm:1QI}
There exists a constant $C> 0$ such that the natural embedding $\Psi: \mathcal{SW}_N \to S$
defined by $\Psi(z_Q)=(c_Q,\eta^{\mathcal h(Q)})$ for every $z_Q \in \mathcal{SW}_N$ is a $(1,C)$-quasi-isometry. Namely, 
 for all vertices
$z_{Q},\,z_{R}\in \mathcal{SW}_N$, one has
\[
\bigl|\, d_S(\Psi(z_{Q}),\,\Psi(z_{R})) 
      \;-\; d_{\mathcal{SW}_N}(z_{Q},z_{R}) \,\bigr|
\;\le\; C
\]
and $\sup_{x \in S}d_S(x,\Psi(\mathcal{SW}_N))< \infty$.
\end{theorem}
\begin{proof}

Let $Q \in \mathscr D_k$, $R \in \mathscr D_j$, $p=\Psi(z_Q)$ and $q=\Psi(z_R)$, with $Q\neq R$.
Let $\gamma$ be a standard geodesic in $\mathcal{SW}_N$ connecting $z_Q$ and $z_R$ and $h=h(\gamma)$ the  height of $\gamma$.

\textbf{Case 1:} suppose there are no inclusion relations between $Q$ and $R$. \\
Observe that by Theorem \ref{thm:normal-form} and Proposition \ref{prop:hmax} 
\begin{align}\label{distSW}
    d_{\mathcal{SW}_N}(z_Q,z_R)= h-k+h-j+O(1)=2h-(k+j)+O(1).
\end{align}
Set $d_N=d_N(c_Q,c_R)$. We have proved in Proposition \ref{prop:hmax} that
\begin{equation}\label{eq:dN-scale}
d_N \;\simeq\; \eta^h.
\end{equation}
By \eqref{eq:hyperbolic-formula} we have that
\begin{align}\label{distance-DR}
 d_S(p,q)=
\frac{1}{\log\eta}\mathrm{arcosh}\left(1+
\frac{d_N^2 + (\eta^k-\eta^j)^2}{2\cdot \eta^{k+j}}\right).
\end{align}
There are two regimes. Let $\widetilde{C}>0$ be a fixed constant.

\smallskip
Assume $d_N\le \widetilde C\eta^{\max(k,j)}$.
Then the numerator in \eqref{distance-DR} is dominated  by $\widetilde C\eta^{2\max(k,j)}$, while the denominator is
$\eta^{k+j}$.  Hence
\[
 d_S(p,q)=\ \log_\eta(\eta^{\,|k-j|})+O(1)=|k-j|+O(1).
\]
On the discrete side, the assumption on $d_N$ implies $\eta^{h} \le C\eta^{\max\{k,j\}}$, and since $h \ge \max\{k,j\}$ it in turn implies $h={\max\{k,j\}}+O(1)$. We conclude by \eqref{distSW} that
\[
d_{\mathcal{SW}_N}(z_{Q},z_{R}) = |k-j|+O(1).
\]
Suppose now that  $d_N\ge \widetilde C \eta^{\max(k,j)}$.
Using~\eqref{eq:dN-scale}, the horizontal term dominates in \eqref{distance-DR}:
\[
 d_S(p,q)
=\log_\eta\left(
\frac{\eta^{2h}}{\eta^{k+j}}\right)+O(1)
=
2h-(k+j)+O(1).
\]
This coincides with \eqref{distSW}.

\textbf{Case 2:} if $Q \subset R$ or $R \subset Q$, then $d_{\mathcal{SW}_N}(z_{Q},z_{R})=|j-k|$ (the unique geodesic connecting them is vertical). This inclusion relation implies that $d_N(c_Q,c_{R}) \le c_2 \eta^{\max\{k,j\}}$, which by \eqref{distance-DR} yields
\begin{align*}
 d_S(p,q)
&=\log_\eta \left(\eta^{2\max\{k,j\}-(k+j)}\right)+O(1)\\&=|k-j|+O(1).
\end{align*}
This concludes the proof.
\end{proof}
The next proposition shows that $\mathcal{SW}_N$ is Gromov hyperbolic. This fact is not required for our main application, as the existence of standard geodesics with uniformly bounded horizontal components on $\mathcal{SW}_N$ has already been established. Interestingly, our argument reverses the strategy of \cite{CMS}: rather than deriving the existence of such geodesics from hyperbolicity, we use their existence to prove the Gromov hyperbolicity of $\mathcal{SW}_N$.

\begin{proposition}\label{gromov} The spiderweb $\mathcal{SW}_N$ is $\delta$-Gromov hyperbolic for some $\delta>0$.
\end{proposition}

\begin{proof}
We identify $\mathcal{SW}_N$ with the associated metric graph obtained
by replacing every edge with a segment of length one.

Recall that a geodesic triangle with sides
$\gamma_{xy}$, $\gamma_{yz}$, and $\gamma_{xz}$ is called
$\delta$-slim if each side is contained in the $\delta$-neighbourhood
of the union of the other two; equivalently,
\[
    \gamma_{xy}
    \subset
    \mathcal N_\delta(\gamma_{xz}\cup\gamma_{yz}),
\]
and the analogous inclusions hold cyclically. A geodesic metric space
is Gromov hyperbolic if all its geodesic triangles are $\delta$-slim
for some uniform constant $\delta\geq0$.

By Lemma~\ref{lem:standard-geodesic-approximation}, it is enough to prove that triangles whose sides
are standard geodesics are uniformly slim. Let
$x,y,z\in\mathcal{SW}_N$, and let
$\gamma_{xy}$, $\gamma_{yz}$, and $\gamma_{xz}$ be standard geodesics
joining the corresponding pairs of vertices. Set
\[
h_{xy}=h(\gamma_{xy}),
\qquad
h_{yz}=h(\gamma_{yz}),
\qquad
h_{xz}=h(\gamma_{xz}).
\]
We first observe that, for any two vertices $a,b\in\{x,y,z\}$, by Proposition \ref{prop:hmax}
\[
h(\gamma_{ab})=
\log_\eta
\left(
d_N(c_{Q_a},c_{Q_b})
\right)
+
O(1),
\]
provided there are no inclusion relations between $Q_a$ and $Q_b$. In this case by \eqref{inclusions} also $\eta^{\max\{\mathcal{h}(a),\mathcal{h}(b)\}}\le C d_N(c_{Q_a},c_{Q_b})$.
If, instead, one of the two cubes is contained in the other, then
\[
h(\gamma_{ab})=\max\{\mathcal h(a),\mathcal h(b)\},
\]
while by \eqref{inclusions}
\[
d_N(c_{Q_a},c_{Q_b})
\le C
\eta^{\max\{\mathcal h(a),\mathcal h(b)\}},
\]
and the same estimate follows.
In any case, $$h(\gamma_{ab})=\log_{\eta}\left(d_N(c_{Q_a},c_{Q_b})+ \eta^{\mathcal{h}(a)}+\eta^{\mathcal{h}(b)}\right)
+
O(1).$$
Set
\[
R_{ab}=d_N(c_{Q_a},c_{Q_b})
+
\eta^{\mathcal h(a)}
+
\eta^{\mathcal h(b)}.
\]
By the triangle inequality in $(N,d_N)$,
\[
R_{xz}
\le
R_{xy}+R_{yz}
\le
2\max\{R_{xy},R_{yz}\}.
\]
Consequently,
\[
h_{xz}
\le
\max\{h_{xy},h_{yz}\}+C,
\]
and the analogous estimates hold after permuting $x,y,z$.

Assume, without loss of generality, that $h_{xz}\ge h_{yz}\ge h_{xy}.$
It follows that
\begin{align}\label{diffh}
0\le h_{xz}-h_{yz}\le C.
\end{align}

Let $K_N>0$ be a uniform upper bound for the length of the horizontal component of every standard geodesic, as provided by Theorem \ref{thm:normal-form}.

We first consider $\gamma_{xy}$. The vertical component of $\gamma_{xy}$ issuing from $x$ is contained in the corresponding vertical component of $\gamma_{xz}$, while the vertical component issuing from $y$ is contained in the corresponding vertical component of $\gamma_{yz}$. Since the horizontal component of $\gamma_{xy}$ has length at most $K_N$, we obtain $\gamma_{xy}
\subset
\mathcal N_{K_N}(\gamma_{xz}\cup\gamma_{yz}).$

We next prove that $\gamma_{xz}
\subset
\mathcal N_C(\gamma_{xy}\cup\gamma_{yz})$ for an appropriate $C>0$.
Let $\gamma_{xz}^{(x)}$ and $\gamma_{xz}^{(z)}$ denote the two vertical components of $\gamma_{xz}$ issuing from $x$ and $z$, respectively, and let $\gamma_{xz}^{(h)}$ denote its horizontal component. Thus
\[
\gamma_{xz}=\gamma_{xz}^{(x)}
\cup
\gamma_{xz}^{(h)}
\cup
\gamma_{xz}^{(z)}.
\]
The vertical components of $\gamma_{xz}$ and $\gamma_{yz}$ issuing from $z$ coincide up to height $h_{yz}$ and  by \eqref{diffh}
the remaining portion of $\gamma_{xz}^{(z)}$ has uniformly bounded length. Hence $\gamma_{xz}^{(z)}
\subset
\mathcal N_C(\gamma_{yz}).$
Moreover, $\gamma_{xz}^{(h)}$ has length at most $K_N$ and intersects $\gamma_{xz}^{(z)}$. Therefore, after possibly enlarging $C$, $\gamma_{xz}^{(h)}
\subset
\mathcal N_C(\gamma_{yz}).$

It remains to consider $\gamma_{xz}^{(x)}$. For every integer $k$ satisfying $\mathcal h(x)\le k\le h_{xy},$
the vertex $p^{k-\mathcal h(x)}(x)$
belongs to the vertical component of $\gamma_{xy}$ issuing from $x$.

Suppose now that $h_{xy}\le k\le h_{yz}.$
Projecting the horizontal component of $\gamma_{xy}$ from height $h_{xy}$ to height $k$ and using the spiderweb property yields a horizontal path of length at most $K_N$ joining $p^{k-\mathcal h(x)}(x)$ and $p^{k-\mathcal h(y)}(y)$,
 see Figure \ref{fig:3}.
The latter vertex belongs to the vertical component of $\gamma_{yz}$ issuing from $y$. Therefore
\[
d_{\mathcal{SW}_N}
\left(
p^{k-\mathcal h(x)}(x),
\gamma_{yz}
\right)
\le K_N.
\]

Finally, the portion of $\gamma_{xz}^{(x)}$ corresponding to heights
$h_{yz}\le k\le h_{xz} $ has length at most $h_{xz}-h_{yz}\le C$
from the vertex at height $h_{yz}$, which is already at uniformly bounded distance from $\gamma_{yz}$. Consequently, $\gamma_{xz}^{(x)}
\subset
\mathcal N_C(\gamma_{xy}\cup\gamma_{yz}),$
and hence
\[
\gamma_{xz}
\subset
\mathcal N_C(\gamma_{xy}\cup\gamma_{yz}).
\]
The inclusion $\gamma_{yz}\subset \mathcal{N}_C(\gamma_{xy}\cup \gamma_{xz})$ is completely analogous. 

We have therefore proved that every triangle whose sides are standard geodesics is $\delta$--slim for some uniform constant $\delta>0$. By the reduction explained above this concludes the proof.
\end{proof}
\begin{figure}
\begin{tikzpicture}[scale=1]

\def\hxy{1.6}
\def\hyz{3.3}
\def\hxz{5.0}

% Base points
\coordinate (x) at (0,0);
\coordinate (y) at (5,0);
\coordinate (z) at (10,0);

% Left branch (allineato)
\coordinate (A) at (1.3,\hxy);
\coordinate (C) at (2.6,\hyz);
\coordinate (E) at (3.9,\hxz);

% Right branch (allineato)
\coordinate (B) at (8.7,\hxy);
\coordinate (D) at (7.4,\hyz);
\coordinate (F) at (6.1,\hxz);

% Central branch above y
\coordinate (Y1) at (5,\hxy);
\coordinate (Y2) at (5,\hyz);

%--------------------------------------------------
% Height guides
%--------------------------------------------------

\draw[dashed,gray!60] (-1,\hxy)--(11,\hxy);
\draw[dashed,gray!60] (-1,\hyz)--(11,\hyz);
\draw[dashed,gray!60] (-1,\hxz)--(11,\hxz);

%--------------------------------------------------
% Branch above x
%--------------------------------------------------

\draw[
blue,
line width=1.6pt,
dash pattern=on 4pt off 4pt
]
(x)--(A);

\draw[
green!60!black,
line width=1.6pt,
dash pattern=on 4pt off 4pt,
dash phase=4pt
]
(x)--(A);

%--------------------------------------------------
% Branch above y (UNICO)
%--------------------------------------------------

\draw[
blue,
line width=1.6pt,
dash pattern=on 4pt off 4pt
]
(y)--(Y1);

\draw[
red,
line width=1.6pt,
dash pattern=on 4pt off 4pt,
dash phase=4pt
]
(y)--(Y1);

%--------------------------------------------------
% Branch above z
%--------------------------------------------------

\draw[
red,
line width=1.6pt,
dash pattern=on 4pt off 4pt
]
(B)--(z);
\draw[
red,
line width=1.6pt,
dash pattern=on 4pt off 4pt
]
(B)--(D);

\draw[
green!60!black,
line width=1.6pt,
dash pattern=on 4pt off 4pt,
dash phase=4pt
]
(B)--(z);
\draw[
green!60!black,
line width=1.6pt,
dash pattern=on 4pt off 4pt,
dash phase=4pt
]
(B)--(D);

%--------------------------------------------------
% gamma_xy
%--------------------------------------------------

\draw[blue,very thick]
(A)--(Y1);

\draw[red,very thick]
(Y2)--(Y1);
%--------------------------------------------------
% gamma_yz
%--------------------------------------------------

\draw[red,very thick]
(Y2)--(D);

%--------------------------------------------------
% gamma_xz
%--------------------------------------------------

\draw[green!60!black,very thick]
(A)--(C)--(E);

\draw[green!60!black,very thick]
(E)--(F);

\draw[green!60!black,very thick]
(F)--(D);

%--------------------------------------------------
% Points
%--------------------------------------------------

\foreach \P in {x,y,z,A,D,E,F,Y1,Y2}
{
  \fill (\P) circle (2.3pt);
}

%--------------------------------------------------
% Labels
%--------------------------------------------------

\node[below] at (x) {$x$};
\node[below] at (y) {$y$};
\node[below] at (z) {$z$};

\node[blue]
at (3.2,\hxy+0.3)
{$\gamma_{xy}$};

\node[red]
at (6.8,\hyz+0.3)
{$\gamma_{yz}$};

\node[green!60!black]
at (5,\hxz-0.35)
{$\gamma_{xz}$};

\node[left] at (-1,\hxy)
{$h(\gamma_{xy})$};

\node[left] at (-1,\hyz)
{$h(\gamma_{yz})$};

\node[left] at (-1,\hxz)
{$h(\gamma_{xz})$};

%--------------------------------------------------
% bounded horizontal part
%--------------------------------------------------

\draw[<->]
(3.9,\hxz+0.5)
--
(6.1,\hxz+0.5);

\node at (5,\hxz+0.85)
{\footnotesize uniformly bounded};

%--------------------------------------------------
% constant C
%--------------------------------------------------

\draw[<->,thick]
(10.8,\hyz)
--
(10.8,\hxz);

\node[right]
at (10.9,{(\hyz+\hxz)/2})
{$\le C$};

\end{tikzpicture}
\begin{caption}{\label{fig:3} Standard geodesics joining $x,y,$ and $z$ described in the proof of Proposition \ref{gromov}.}
\end{caption}
\end{figure}

\begin{corollary}\label{COR:GROMOV}
If $S$ is a geodesic space, then $S$ is Gromov hyperbolic.
\end{corollary}
\begin{proof}
    This follows by putting together the fact that $\mathcal{SW}_N$ is Gromov hyperbolic and that $\mathcal{SW}_N$ is 1-quasi-isometric to $S$ (see e.g. \cite[Ch. III.H, Th. 1.9]{BH}).
\end{proof}
\begin{remark}
    Corollary \ref{COR:GROMOV} applies in particular to the cuspidal manifolds $\mathrm{Cusp}(X)$ considered by H.-Q. Li \cite{LiJFA, LiMa, Li2007} when the base space $X$ is doubling.
\end{remark}

\section{Transference and main results}\label{sec:transf}

Let $S = N \times \mathbb R_{+}$, where $(N,d_N,\mu_N)$ is a geodesic doubling metric measure space. 
For every $(x,y),(x',y') \in S$ recall that
\begin{align*}
 d_S\big((x,y),(x',y')\big)
 =
 C_N \,\mathrm{arcosh}\!\left(
 1 + \frac{d_N(x,x')^2 + (y-y')^2}{2yy'}
 \right),
\end{align*}
where $C_N>0$ is the normalisation constant introduced above.
The precise value $C_N=(\log \eta)^{-1}$ plays no role in our results. Indeed, the validity of a global Poincaré inequality is invariant under rescaling of the metric: if such an inequality holds for $d_S$, then it also holds for any constant multiple of $d_S$, possibly with different constants. We proved in Theorem \ref{thm:1QI} that there is a graph $\mathcal{SW}_N$ that is $1$--quasi--isometric to $S$. 

\begin{definition}\label{def:comp}
Let $(M,d_M,\rho)$ be a metric measure space and let
$(\mathcal{G},d_{\mathcal{G}},\mu)$ be a graph endowed with a measure $\mu$.
Assume that the vertex set of $\mathcal{G}$ is identified with a uniformly
dense and uniformly separated subset of $M$. We say that $\rho$ and $\mu$ are \emph{compatible} if there exists a constant
$C\geq 1$ such that
\[
C^{-1}\rho(B_M(v,1))
\leq
\mu(v)
\leq
C\,\rho(B_M(v,1))
\]
for every vertex $v\in\mathcal{G}$.
\end{definition}
In this section, whenever the set of vertices of $\mathcal G$ is a uniformly dense and uniformly separated subset of a metric space $(M,d_M)$ and the two spaces are 1-quasi-isometric, we denote by $D,d,k_2$ the positive constants such that

\begin{align}\label{Condition 1 of 1-q isom}
		 d_\mathcal{G}(x,y) -k_2 \leq d_{M}(x,y) \leq  d_\mathcal{G}(x,y) +k_2, \quad \forall x,y \in \mathcal{G},
	\end{align}
	
	\begin{align}\label{distinct vertices are distant also in S}
		\sup_{s \in M}d_{M}(s, \mathcal{G}) = D < \infty,
	\end{align}
and  \begin{align}\label{disjointballs}2d=\inf_{u,v \in \mathcal{G}, u\ne v} d_M(u,v)>0.\end{align}
Our transference result is the following.
\begin{theorem}\label{transference}
    Suppose that $(M,d_M,\rho)$ and $(\mathcal G,d_{\mathcal G}, \mu)$  are locally doubling metric measure spaces and that $\mathcal{G}$ is a graph equipped with the discrete distance $d_{\mathcal G}$. Assume the following
    \begin{itemize}
        \item[i)]the set of vertices of $\mathcal G$ is a uniformly dense and uniformly separated subset of $M$;
        \item[ii)] there is a $1$-quasi-isometry between $\mathcal{G}$ and $M$;
        \item[iii)] $(\mathcal{G},d_{\mathcal{G}},\mu)$ supports the following  {\it global} $L^p$-Poincaré inequality for some $p \in [1,\infty)$, i.e., there exists $C>0$ such that for every $R>0$
        \begin{align*}
            \sum_{y \in B_R} |F(y)-F_{B_R}|^p \mu(y) \le C R^p \sum_{y \in B_R} |\nabla F(y)|^p \mu(y),
        \end{align*}
        and $(M,d_M,\rho)$ supports a weak local $L^p$-Poincaré inequality up to scale $R_{\mathrm{loc}}>2D+k_2$ for some dilation parameter $\lambda\ge1$;
        \item[iv)] $\rho$ and $\mu$ are compatible. 
\end{itemize} Then, $(M, d_M, \rho)$ supports a global $L^p$-Poincaré inequality. Moreover, there is a constant $\mathcal{C}_0=(\lambda+2)k_2 + (2\lambda+3)D>0$ such that for every $R>R_\mathrm{loc}$
\begin{align*}
    \int_{B_M(x_0,R)} |f-f_{B_M(x_0,R)}|^p \mathrm{d}\rho \le C R^p \int_{B_M(x_0,R+\mathcal{C}_0)} |g_f|^p \mathrm{d}\rho,
\end{align*} for every upper gradient $g_f$ of $f$. 
\end{theorem}
Before we prove the theorem we need the following technical lemma.
\begin{lemma}\label{lem:measball}
    Under the hypotheses of Theorem \ref{transference} the following holds: for every $R>1$   \begin{align}\label{4.31}
        \mu(B_\mathcal{G}(v,R)) \simeq \rho(B_M(v,R)) \qquad \forall v\in \mathcal{G}.
    \end{align}
      Moreover, for every $x  \in M$, $R >1$ and $K>0$ we have
      \begin{align}\label{4.32}
     \rho(B_M(x,R+K)) \simeq_{K} \rho(B_M(x,R)).
    \end{align}
\end{lemma}
\begin{proof}
Assume without loss of generality that $R \in \mathbb N$. First notice that
    \begin{align}\label{RR+1}
         \mu(B_\mathcal{G}(v, R+1))= \mu(B_\mathcal{G}(v, R))+\sum_{y \in S_\mathcal{G}(v,R+1)} \mu(y) \le C\mu(B_\mathcal{G}(v, R)),
    \end{align}
   where $S_{\mathcal{G}}(v,R+1)$ denotes the sphere centered at $v$ with radius $R+1$ and we have used that, since $\mu$ is locally doubling, every vertex has at most a fixed number of neighbours and that $\mu(y) \simeq \mu(x)$ if $x \sim y$.

    Moreover, for every $R>k_2$ and $v_0 \in \mathcal{G}$  we have
\begin{align*}
   B_\mathcal{G}(v_0, R-k_2) \subset B_M(v_0, R) \subset \bigcup_{v \in B_\mathcal{G}(v_0, R+D+k_2)} B_M(v, D).
\end{align*}
Thus, for every $R>k_2+d$,
\begin{align*} \mu(B_\mathcal{G}(v_0, R-k_2-d))&\le \sum_{y \in B_M(v_0, R-d)\cap \mathcal{G}} \mu(y) \\&\le C   \sum_{y \in B_M(v_0, R-d) \cap \mathcal{G}} \rho(B_M(y,d)) \\&\le C  \rho(B_M(v_0, R)) \\
&\le C\sum_{y\in B_\mathcal{G}(v_0, R+D+k_2)}  \rho(B_M(y, D)) \\
&\le C \sum_{y\in B_\mathcal{G}(v_0, R+D+k_2)}  \mu(y)
\\&\le C \mu(B_\mathcal{G}(v_0, R+D+k_2)),
\end{align*} where we have used that $\{B_M(y,d)\}_{y \in \mathcal{G}}$ are disjoint and that $\rho$ and $\mu$ are compatible.
    Since the measures in the left and right hand side are comparable by \eqref{RR+1}, \eqref{4.31} follows. Moreover, for  every $K>0$ $$\rho(B_M(v_0,R))\simeq_K\rho(B_M(v_0,R+K))\qquad \forall R>k_2+d.$$ 
The remaining range $1<R\leq k_2+d$ follows from the local doubling
property.
To derive \eqref{4.32} for an arbitrary center $x \in M$, it suffices to use the fact that $\mathcal{G}$ is a discretisation of $M$. 
\end{proof}
\begin{proof}[Proof of Theorem \ref{transference}]
Let $k_2,D$ and $2d$ as in \eqref{Condition 1 of 1-q isom}, \eqref{distinct vertices are distant also in S}, and \eqref{disjointballs}. Up to replacing $D$, we can assume without loss of generality that $D\ge 1.$
    
Given $s \in M$, we denote by $\overline{s} \in \mathcal{G}$ a vertex such that $d_{M}(s,\overline{s}) \leq D.$ We can assume  without loss of generality that the radius of the ball is $r \ge 1$. We have that for every $s\in M$
	\begin{align*}
		B_M(s,r) \subset B_M(\overline{s},r+D), \qquad 	B_M(\overline{s},r) \subset B_M(s,r+D).
	\end{align*}
	Conditions $i)$ and $ii)$ imply that there is a constant $k_2>0$ such that if $v \in \mathcal{G}$ and $R >k_2$,
	\begin{align}\label{Containments of balls in SW}
		 B_\mathcal{G}(v,R-k_2) \subset \mathcal{G} \cap B_M(v,R) \subset 	 B_\mathcal{G}(v,R+k_2).
	\end{align}
	Let $x \in M$  and set $B_M= B_M(x,r)$. Define
	\begin{align*}
V_{B_M} = \set{v \in \mathcal{G} \colon d_M(v,B_M) \leq D}.
	\end{align*}
	We have 
	\begin{align}\label{Inequ BS in union of Bv1}
		B_M \subset \bigcup_{v \in V_{B_M}} \overline{B_M(v,D)}.
	\end{align}
	Indeed, if $z \in 	B_M$, then there is  $\overline{z} \in V_{B_M}$ such that $z \in B_M(\overline{z},D)$.
  Moreover, for every $v \in V_{B_M}$
	\begin{align}\label{dist2 of v from x leq r+1}
		d_M(x,v) \leq r+D.
	\end{align} Moreover, for every  $y \in B_M$
    \begin{align*}
        d_M(x,v) \le d_M(x,y)+d_M(y,v) \le r+d_M(v,y),
    \end{align*} and $\inf_{y\in B_M} d_M(y,v) \le D$ yields \eqref{dist2 of v from x leq r+1}.
	Hence, for every $v \in V_{B_M}$ the triangle inequality implies
	\begin{align}\label{containment epsilon}
	B_M(v,\epsilon)\subset B_M(x,r+D+\epsilon) \quad \forall \epsilon >0,
	\end{align}
	and
	\begin{align}\label{union depending on epsilon}
		\bigcup_{v \in V_{B_M}} B_M(v,\epsilon) \subset B_M(x,r+D+\epsilon) \quad \forall \epsilon>0.
	\end{align}
By the local doubling property of $(M,d_M,\rho)$  and a standard covering argument for discretisations in locally doubling metric measure spaces, for every $r>0$ there exists $\kappa_r>0$ such that
	\begin{align}\label{Capital M Local doubl S}
		\sum_{v \in \mathcal{G}} \mathds{1}_{B_M(v, r)} (z)\leq \kappa_r\quad \forall z\in M.
	\end{align}
By Jensen's inequality and by \eqref{Inequ BS in union of Bv1},
	\begin{align}\label{Inequality at beginning of proof of glob Poinc}
	 \nonumber	\int_{B_M} |f - f_{B_M}|^p \mathrm{d}\rho  
	 	&\leq \int_{B_M} \int_{B_M} |f(x) - f(y)|^p \mathrm{d}\rho(x) \frac{\mathrm{d}\rho(y)}{\rho(B_M)}\\ &\leq C \sum_{v \in V_{B_M}} \sum_{w \in V_{B_M}} \int_{B_M(v,D)} \int_{B_M(w,D)} |f(x)-f(y)|^p \mathrm{d}\rho(x) \frac{\mathrm{d}\rho(y)}{\rho(B_M)}.
	 	\end{align}
	 	For $v,w \in V_{B_M}$,
	 	\begin{align*}
	\int_{B_M(v,D)}& \int_{B_M(w,D)} |f(x)-f(y)|^p \mathrm{d}\rho(x) \frac{\mathrm{d}\rho(y)}{\rho(B_M)} \\
	&\leq   2^{2p-2}\int_{B_M(v,D)} \int_{B_M(w,D)} |f(x) - f_{B_M(w,D)}|^p\mathrm{d}\rho(x) \frac{\mathrm{d}\rho(y)}{\rho(B_M)}\\
	 	&+ 2^{2p-2}\int_{B_M(v,D)} \int_{B_M(w,D)} |f(y) - f_{B_M(v,D)}|^p\mathrm{d}\rho(x) \frac{\mathrm{d}\rho(y)}{\rho(B_M)}\\
	 	&+ 2^{2p-2} \int_{B_M(v,D)} \int_{B_M(w,D)}  |f_{B_M(v,D)} - f_{B_M(w,D)}|^p\mathrm{d}\rho(x) \frac{\mathrm{d}\rho(y)}{\rho(B_M)}\\
	 	& = 2^{2p-2}  \frac{\rho(B_M(v,D))}{\rho(B_M)} \int_{B_M(w,D)} |f(x) - f_{B_M(w,D)}|^p\mathrm{d}\rho(x)\\
	 	&+  2^{2p-2} \frac{\rho(B_M(w,D))}{\rho(B_M)}  \int_{B_M(v,D)} |f(y) - f_{B_M(v,D)}|^p\mathrm{d}\rho(y) \\
	 	&+ 2^{2p-2} |f_{B_M(v,D)} - f_{B_M(w,D)}|^p \frac{\rho(B_M(v,D))  \rho(B_M(w,D))}{\rho(B_M)}.
	 \end{align*}
	 By  $iii)$, for every $w  \in V_{B_M}$,
	 \begin{align*}   \int_{B_M(w,D)} |f(x) - f_{B_M(w,D)}|^p\mathrm{d}\rho(x) \leq C   \int_{B_M(w,\lambda D)} |g_f |^p\mathrm{d}\rho,
	\end{align*}
which implies by \eqref{Inequality at beginning of proof of glob Poinc} that
\begin{align}\label{Inequality second in proof of Poinc}
\nonumber	\int_{B_M} |f - f_{B_M}|^p \mathrm{d}\rho   &\leq C2^{2p-1} \frac{\sum_{v \in V_{B_M}}\rho(B_M(v,D))}{\rho(B_M)} \sum_{w \in V_{B_M}}   \int_{B_M(w,\lambda D)} |g_f |^p\mathrm{d}\rho \\
\nonumber	&+ C2^{2p-2} \sum_{v \in V_{B_M}} \sum_{w \in V_{B_M}}   |f_{B_M(v,D)} - f_{B_M(w,D)}|^p \frac{\rho(B_M(v,D))  \rho(B_M(w,D))}{\rho(B_M)} \\ 
&\le I+II. 
\end{align}
By  the local doubling property on $M$, there exists $C>0$ such that for every vertex $v$
\begin{align*}
	 \rho(B_M(v,D)) \leq C \rho(B_M(v,d)).
\end{align*}
Then, since the $B_M(v,d)$ are disjoint, by \eqref{union depending on epsilon} with $\epsilon = d$
\begin{align}\label{Inequality sum of meas wrt meas of ball}
	 \frac{\sum_{v \in V_{B_M}}\rho(B_M(v,D))}{\rho(B_M)} & \leq  C\frac{\sum_{v \in V_{B_M}}\rho(B_M(v,d))}{\rho(B_M)} \\&\nonumber\leq  C\frac{\rho\left(\bigcup_{v \in V_{B_M}} B_M(v,d)\right)}{ \rho(B_M)}\\
	\nonumber &\leq  C\frac{\rho(B_M(x, r+d+ D))}{  \rho(B_M)} \\&\nonumber\leq C,
\end{align} by Lemma \ref{lem:measball}.
Moreover, by \eqref{containment epsilon} with $\epsilon =\lambda D$ and by \eqref{Capital M Local doubl S}
\begin{align*}
	 I&\le C\sum_{w \in V_{B_M}}   \int_{B_M(w,\lambda D)} |g_f |^p\mathrm{d}\rho \\&=  C\sum_{w \in V_{B_M}}   \int_{B_M(x,r+2\lambda D)} \mathds{1}_{B_M(w,\lambda D)}  |g_f |^p\mathrm{d}\rho \\&=C\int_{B_M(x,r+2\lambda D)} \sum_{w \in V_{B_M}}    \mathds{1}_{B_M(w,\lambda D)}  |g_f |^p\mathrm{d}\rho \\
	  &\leq C\kappa_{\lambda D} \int_{B_M(x,r+2\lambda D)}   |g_f |^p\mathrm{d}\rho.
\end{align*}
Thus, to conclude the proof it suffices to bound $II$. \\ Define the function $F(v)= f_{B_M(v,D)}$ for $v \in \mathcal{G}
	$ and 	set  $\mathcal{m}= F_{B_\mathcal{G}(\overline{x}, r+2D+k_2)}$. Then
	\begin{align*}
		 |f_{B_M(v,D)} - f_{B_M(w,D)}|^p &= |F(v) - F(w)|^p \leq 2^{p-1} (|F(v) - \mathcal{m}|^p + | F(w) - \mathcal{m}|^p).
	\end{align*}
 	Since for every vertex $z$ we have  $\rho(B_M(z,D)) \simeq \mu(z)$ uniformly in $z$,  by \eqref{Inequality sum of meas wrt meas of ball}
	\begin{align}\label{Inequality most difficult term in Poinc proof}
        \nonumber II&\leq C2^p  	\sum_{v \in V_{B_M}}  |F(v) - \mathcal{m}|^p \rho(B_M(v,D))   \sum_{w \in V_{B_M}}  \frac{ \rho(B_M(w,D))}{\rho(B_M)} \\
		&\leq C2^p  \sum_{v \in V_{B_M}}  |F(v) - \mathcal{m}|^p \mu(v).
	\end{align}
	Recall that $d_M(\overline{x},x) \le D$, hence, for every $v \in V_{B_M}$, by \eqref{Condition 1 of 1-q isom}, \eqref{dist2 of v from x leq r+1} 
	\begin{align*}
	d_\mathcal{G}(v,\overline{x}) &\leq d_M(v, \overline{x}) + k_2 \leq d_M(v,x) + d_M(x,\overline{x})+k_2\le r+2D+k_2.
	\end{align*} It follows that
    \begin{align*}
		V_{B_M} \subset B_\mathcal{G}(\overline{x},  r+2D+k_2).
	\end{align*}
Thus, by assumption $iii)$,
\begin{align}\label{(A)}
	\nonumber\sum_{v \in V_{B_M}}  |F(v)-\mathcal{m}|^p \mu(v)& \leq 	\sum_{v \in B_\mathcal{G}(\overline{x},  r+2D+k_2)}  |F(v)-\mathcal{m}|^p \mu(v) \\
	&\leq C (r+2D+k_2)^p \sum_{v \in B_\mathcal{G}(\overline{x},   r+2D+k_2 )} |\nabla F(v) |^p \mu(v),
\end{align}
where 
\begin{align*}
	|\nabla F(v) | = \sum_{y \sim v} |F(v) - F(y)|.
\end{align*}
Given $y \sim v$, for every $c \in \mathbb C$,
\begin{align*}
 |F&(v) - F(y)|\\ &\leq \abs{\frac{1}{\rho(B_M(v,D))} \int_{B_M(v,D)} f \mathrm{d}\rho - \frac{1}{\rho(B_M(y,D))} \int_{B_M(y,D)} f \mathrm{d}\rho  } \\
 &= \abs{\frac{1}{\rho(B_M(v,D))} \int_{B_M(v,D)} (f -c)\mathrm{d}\rho - \frac{1}{\rho(B_M(y,D))} \int_{B_M(y,D)} (f-c) \mathrm{d}\rho  }\\
 &\leq  \frac{1}{\rho(B_M(v,D))} \int_{B_M(v,D)} |f -c|\mathrm{d}\rho +  \frac{1}{\rho(B_M(y,D))} \int_{B_M(y,D)} |f-c| \mathrm{d}\rho\\
 & \leq  \frac{C}{\rho(B_M(v,D))} \left(  \int_{B_M(v,D)} |f -c|\mathrm{d}\rho +   \int_{B_M(y,D)} |f-c| \mathrm{d}\rho \right)\\
 &\leq \frac{2C}{\rho(B_M(v,D))} \int_{B_M(v,2D+k_2)} |f-c|\mathrm{d}\rho,
\end{align*}
	 since $y \sim v$ implies that $\rho(B_M(y,D)) \simeq \rho(B_M(v,D))$  uniformly in $v$ and that $B_M(y,D) \subset B_M(v,2D+k_2)$  by \eqref{Condition 1 of 1-q isom}. Finally,  using that $\mathcal{G}$ has bounded geometry,  
	 \begin{align*}
	 	|\nabla  F(v) |  &\leq \frac{C}{\rho(B_M(v,D))} \int_{B_M(v,2D+k_2)} |f-c|\mathrm{d}\rho.
	 \end{align*}
	 By Hölder's inequality, 
	 \begin{align*}
	 		|\nabla  F(v) |^p  \mu(v) &\leq \frac{C\rho(B_M(v,D))^{p/p'}\mu(v)}{\rho(B_M(v,D))^p} \int_{B_M(v,2D+k_2)} |f-c|^p\mathrm{d}\rho\\ &= \frac{C  \mu(v)}{\rho(B_M(v,D))} \int_{B_M(v,2D+k_2)} |f-c|^p\mathrm{d}\rho \\
	 		&\leq C \int_{B_M(v,2D+k_2)} |f-c|^p\mathrm{d}\rho,
	 \end{align*}
	  since $\rho(B_M(v,D)) \simeq \mu(v)$.  If we choose $c= f_{B_M(v,2D+k_2)}$, by $iii)$,
	 \begin{align}\label{(B)}
	 |\nabla F(v) |^p \mu(v) \leq    C \int_{B_M(v,\lambda(2D+k_2))} |g_f|^p\mathrm{d}\rho.
	 \end{align}
	Thus, \eqref{(A)} and \eqref{(B)} imply
	\begin{align*}
		 	II&\leq C r^p \sum_{v \in B_\mathcal{G}(\overline{x},  (r+2D+k_2))}   \int_{B_M(v,\lambda(2D+k_2))} |g_f|^p\mathrm{d}\rho.
	\end{align*}
We have that
	\begin{align}\label{last containment of balls}
		 \bigcup_{v \in B_\mathcal{G}(\overline{x},  r+2D+k_2 )} B_M(v,\lambda(2D+k_2))  \subset B_M(x, r +(\lambda+2)k_2 + (2\lambda+3)D).
	\end{align}
	In fact, if $v \in B_\mathcal{G}(\overline{x},  r+2D+k_2)$ and $z \in B_M(v,\lambda(2D+k_2))$, then
	\begin{align*}
		d_M(z,x) &\leq d_M(z,v) + d_M(v,x) \\&\leq \lambda(2D+k_2) +d_M(v,\overline{x}) + d_M(\overline{x},x) \\&\leq \lambda(2D+k_2) +D + d_M(v,\overline{x}).
	\end{align*}
	By \eqref{Condition 1 of 1-q isom},
	\begin{align*}
		d_M(v,\overline{x}) \leq d_\mathcal{G}(v, \overline{x} ) + k_2 \leq r+2D+k_2 +k_2 = r+2D+2k_2,
	\end{align*}
so that \eqref{last containment of balls} follows.
	By \eqref{Capital M Local doubl S},
	\begin{align*}
	& \sum_{v \in B_\mathcal{G}(\overline{x},  r+2D+k_2)}   \int_{B_M(v,\lambda(2D+k_2))} |g_f|^p\mathrm{d}\rho \\
	 &\leq    \int_{  B_M(x, r +(\lambda+2)k_2 + (2\lambda+3)D) }  \sum_{v \in \mathcal{G}}  \mathds{1}_{B(v, \lambda(2D+k_2))}(z)   |g_f(z)|^p\mathrm{d}\rho(z)\\
	 &\leq \kappa_{\lambda(2D+k_2)} \int_{  B_M(x, r  +(\lambda+2)k_2 + (2\lambda+3)D) }   |g_f|^p\mathrm{d}\rho.
	\end{align*}
	For $R\le R_{\mathrm{loc}}$, the desired weak Poincaré inequality follows from the assumed local Poincaré inequality on $M$. For $R>R_{\mathrm{loc}}$, the preceding argument gives the large-scale estimate. Hence $M$ supports a global weak $L^p$-Poincaré inequality.
	 
\end{proof}
We now verify that the hyperbolic-type space $(S,d_S,\rho)$ and the spiderweb $\mathcal{SW}_N$ satisfy
the assumptions of Theorem~\ref{transference}. 
\begin{definition}
    Assume that the measure $\mu_N$ on $N$ is doubling and define on $S=N\times \mathbb R_+$ the measure
\begin{align}\label{defrho}
\mathrm{d}\rho(x,a)=\frac{\mathrm{d}\mu_N(x)\,\mathrm{d}a}{a}.
\end{align}
For every vertex $z_Q\in\mathcal{SW}_N$, set
\begin{align}\label{defmu}
\mu(z_Q)=\rho(B_S(\Psi(z_Q),1)),
\end{align}
where $\Psi(z_Q)=(c_Q,\eta^{\mathcal{h}(Q)})$ is the natural embedding of $\mathcal{SW}_N$ into $S$.
\end{definition}

We begin by studying
the geometry and measure of balls in $S$. In particular, the next
lemma implies that $\rho$ is locally doubling.
\begin{lemma}\label{lem:ball-measure}
Assume that $(N,d_N, \mu_N)$ is a metric measure space, set $S=N \times \mathbb R_{+}$ and equip $S$ with the metric \eqref{eq:hyperbolic-formula} and the measure defined in \eqref{defrho}.  
Then, for every $R>0$, there exist constants
$c_R,C_R>0$, depending only on $R$, such that for every
$(x_0,a_0)\in S$,

\[
B_N(x_0,c_Ra_0)
\times
\bigl(a_0e^{-R/(2C_N)},a_0e^{R/(2C_N)}\bigr)
\subset
B_S((x_0,a_0),R)
\]
and
\[
B_S((x_0,a_0),R)
\subset
B_N(x_0,C_Ra_0)
\times
\bigl(a_0e^{-R/C_N},a_0e^{R/C_N}\bigr),
\]
where $C_N$ is the constant appearing in \eqref{eq:hyperbolic-formula}.
Moreover,

\[
\frac{R}{C_N}\,
\mu_N(B_N(x_0,c_Ra_0))
\le
\rho(B_S((x_0,a_0),R))
\le
\frac{2R}{C_N}\,
\mu_N(B_N(x_0,C_Ra_0)).
\] In particular, if $\mu_N$ is doubling, then $\rho$ is locally doubling.
\end{lemma}

\begin{proof}
Set
\[
\alpha_R=\frac{R}{C_N},
\qquad
\kappa_R=\cosh(\alpha_R)-1.
\]
By \eqref{eq:hyperbolic-formula} $d_S((x,a),(x_0,a_0))<R$  if and only if

\[
d_N(x,x_0)^2
<
2aa_0\kappa_R-(a-a_0)^2.
\]

Hence
\[
B_S((x_0,a_0),R)
=
\Bigl\{
(x,a):
a_0e^{-\alpha_R}<a<a_0e^{\alpha_R},
\;
d_N(x,x_0)<r_R(a)
\Bigr\},
\]
where $r_R(a)^2=2aa_0\kappa_R-(a-a_0)^2.$\\
Writing $a=a_0s$, we obtain

\[
r_R(a)
=
a_0
\sqrt{
2s\kappa_R-(s-1)^2
},
\]
with $s\in (e^{-\alpha_R},e^{\alpha_R}).$ Since
$f_R(s)=
2s\kappa_R-(s-1)^2$
is continuous on $[e^{-\alpha_R/2},e^{\alpha_R/2}]$ and strictly positive there,
there exist constants $c_R,C_R>0$ such that
\[
c_R
\le
\sqrt{f_R(s)}
\le
C_R \qquad \forall s\in[e^{-\alpha_R/2},e^{\alpha_R/2}].
\]
Therefore,
\[
c_Ra_0
\le
r_R(a)
\le
C_Ra_0
\qquad
\forall a\in
[a_0e^{-\alpha_R/2},a_0e^{\alpha_R/2}].
\]
It follows that
\[
B_N(x_0,c_Ra_0)\times
(a_0e^{-\alpha_R/2},a_0e^{\alpha_R/2})
\subset
B_S((x_0,a_0),R).
\]
On the other hand, after possibly enlarging $C_R$, we also have $r_R(a)\le C_Ra_0$ for every
$a\in(a_0e^{-\alpha_R},a_0e^{\alpha_R})$,
 and thus
\[
B_S((x_0,a_0),R)
\subset
B_N(x_0,C_Ra_0)\times
(a_0e^{-\alpha_R},a_0e^{\alpha_R}).
\]
Integrating with respect to $\mathrm da/a$ yields
\[
\frac{R}{C_N}\,
\mu_N(B_N(x_0,c_Ra_0))
\le
\rho(B_S((x_0,a_0),R)) \le
\frac{2R}{C_N}\,
\mu_N(B_N(x_0,C_Ra_0)).
\]
In particular, if $\mu_N$ is doubling, then for every $R>0$
\begin{align*}
\rho(B_S((x_0,a_0),2R))
&\le
\frac{4R}{C_N}\,
\mu_N(B_N(x_0,C_{2R}a_0)) \\&\le C_R'\frac{4R}{C_N}\mu_N(B_N(x_0,c_{2R}a_0)) \\&\le C''_R\,\rho(B_S((x_0,a_0),R)).
\end{align*}
This concludes the proof.
\end{proof}
The next result shows that the
natural measure induced by $\rho$ on the vertices of $\mathcal{SW}_N$ defined in \eqref{defmu}
is a locally doubling quasi-flow.
\begin{corollary}\label{cor:quasi-flow-doubling}
 The measure $\mu$ defined in \eqref{defmu} is a quasi-flow on $\mathcal{SW}_N$. More precisely, there exists a
constant $C\geq 1$ such that for every cube $Q$ and every $m\in\mathbb N$,
\[
C^{-1}\mu(z_Q)
\leq
\sum_{z_R\in s_m(z_Q)} \mu(z_R)
\leq
C\,\mu(z_Q),
\]
where $s_m(z_Q)$ denotes the set of descendants of $z_Q$ of generation $m$. Moreover, $\mu$ is locally doubling.
\end{corollary}

\begin{proof}
Let $k=\mathcal h(z_Q)$ and let $z_R\in s_m(z_Q)$.
By Lemma~\ref{lem:ball-measure},
\[
\mu(z_Q)
=
\rho(B_S(\Psi(z_Q),1))
\simeq
\mu_N(B_N(c_Q,\eta^k)),
\]
and similarly,
\[
\mu(z_R)
\simeq
\mu_N(B_N(c_R,\eta^{k-m})).
\]
Using \eqref{inclusions} and the fact $\mu_N$ is doubling, we deduce that
\[
\mu(z_Q) \simeq \mu_N(B_N(c_Q,\eta^k))
\simeq
\mu_N(Q) \quad \text{and} \quad
\mu(z_R)\simeq\mu_N(B_N(c_R,\eta^{k-m}))
\simeq
\mu_N(R),
\]
with constants independent of $Q$, $R$ and $m$.
Summing over all descendants of generation $m$ yields
\[
\sum_{z_R\in s_m(z_Q)}\mu(z_R)
\simeq
\sum_{\substack{
R\subset Q\\
\mathcal{h}(R)=\mathcal{h}(Q)-m
}}\mu_N(R).
\]
Since the cubes in $\{R\subset Q \ : \ \mathcal{h}(R)=\mathcal{h}(Q)-m
\}$ form a partition of $Q$,
\[
\sum_{\substack{
R\subset Q\\
\mathcal{h}(R)=\mathcal{h}(Q)-m
}}\mu_N(R)
=
\mu_N(Q).
\]
Combining the previous estimates gives
\[
\sum_{\substack{
R\subset Q\\
\mathcal{h}(R)=\mathcal{h}(Q)-m
}}\mu(z_R)
\simeq
\mu_N(Q)
\simeq
\mu(z_Q).
\]
The implied constants depend only on the doubling constant of $\mu_N$
and on the parameters of the dyadic system, and are independent
of $Q$ and $m$.

This proves the quasi-flow property of $\mu$. Since $\mathcal{SW}_N$ has bounded geometry by construction, the local doubling property is equivalent to $\mu(z_Q) \simeq \mu(z_R)$ whenever $z_Q \sim z_R$. Observe that if $z_Q \sim z_R$ then thanks to the  1-quasi-isometric relation $d_S(\Psi(z_Q),\Psi(z_R))\le 1+k_2$, thus $\mu(z_Q)=\rho(B_S(\Psi(z_Q),1)) \simeq \rho(B_S(\Psi(z_R),1))=\mu(z_R),$ concluding the proof. 
\end{proof}
It remains to verify the last assumption of
Theorem~\ref{transference}, namely, we now prove that $S$ supports a local Poincaré inequality.
\begin{lemma}\label{lem:localPI}
Suppose that $(N,d_N,\mu_N)$ is a doubling metric measure space
supporting a global weak $L^p$-Poincar\'e inequality for some $p\in[1,\infty)$. Namely, there exist
$K_N>0$, $\lambda \ge 1$ such that for every ball
$B_N(x,r)$,
every locally integrable function $u$, and every upper gradient $g_u$ of $u$,
\[
\int_{B_N(x,r)}
|u-u_{B_N(x,r)}|^p\,\mathrm{d}\mu_N
\le
K_N r^p
\int_{B_N(x,\lambda r)}
g_u^p\,\mathrm{d}\mu_N.
\]
Then, for every $r_0>0$, $(S,d_S,\rho)$ supports a local $L^p$-Poincar\'e inequality up to scale $r_0$.
More precisely, for every $r_0>0$ there are  $\lambda'\ge1$ and $C>0$ depending on $r_0$ such that for every
ball $B_S(z,r)$ with $r<r_0$,
every locally integrable function $F$ on $S$ and every upper gradient $G$ of $F$,
\begin{align}\label{lemma46}
\int_{B_S(z,r)}
|F-F_{B_S(z,r)}|^p\,\mathrm{d}\rho
\le
Cr^p
\int_{B_S(z,\lambda 'r)}
G^p\,\mathrm{d}\rho.
\end{align}
\end{lemma}

\begin{proof}

Write points of $S$ as $(x,a)$ and introduce logarithmic coordinates $t=\log a.$
Then, the measure becomes $\mathrm{d}\rho(x,a)=\mathrm{d}\mu_N(x)\,\mathrm{d}t.$
Fix a point $z_0=(x_0,a_0)$ and set $t_0=\log a_0$.
By Lemma~\ref{lem:ball-measure}, there exist constants
$c_1,c_2>0$ depending on $r_0$ such that for every $r<r_0$,
\begin{align}\label{rombi}
Q_{c_1r}(z_0)
\subset
B_S(z_0,r)
\subset
Q_{c_2r}(z_0),
\end{align}
where
\[
Q_s(z_0)
=
B_N(x_0,a_0 s)
\times
(t_0-s,t_0+s), \qquad s>0.
\]
Moreover,
\[
\rho(Q_{c_2r}(z_0))
\simeq_{r_0}
\rho(Q_{c_1r}(z_0)).
\]
Therefore it suffices to prove the Poincar\'e inequality on cylinders.\\
Let
$Q
=
B_N(x_0,R)\times I$ where $R=a_0 r$ and $I=(t_0-r,t_0+r)$.
Given $F\in L^1_{\mathrm{loc}}(S)$, for simplicity write $\widetilde F(x,t)=F(x,e^t)$ for every $x \in N$ and $t \in \mathbb R$ and define
\[
\widetilde F_I(x)
=
\frac1{|I|}
\int_I
\widetilde F(x,t)\,\mathrm{d}t \qquad \forall x\in N.
\]
Then
\[
\widetilde F-\widetilde F_Q
=
(\widetilde F-\widetilde F_I)
+
(\widetilde F_I-(\widetilde F_I)_Q),
\]
and hence
\[
|\widetilde F-\widetilde F_Q|^p
\le
2^{p-1}
|\widetilde F-\widetilde F_I|^p
+
2^{p-1}
|\widetilde F_I-(\widetilde F_I)_Q|^p.
\]
Integrating over $Q$ yields
\[
\int_Q
|\widetilde F-\widetilde F_Q|^p
\,\mathrm{d}\mu_N\mathrm{d}t
\le
2^{p-1}(I_1+I_2),
\]
where
\[
I_1
=
\int_Q
|\widetilde F-\widetilde F_I|^p\,\mathrm{d}\mu_N\mathrm{d}t
\]
and
\[
I_2
=
\int_Q
|\widetilde F_I-(\widetilde F_I)_Q|^p\,\mathrm{d}\mu_N\mathrm{d}t.
\]
\noindent
We first estimate  $I_1$. Fix $x\in B_N(x_0,R)$.
Consider the vertical curve
\[
\gamma_x(s)=(x,e^s),
\qquad s\in I.
\]
Observe that $d_S(\gamma_x(s),\gamma_x(t))=C_N\mathrm{arcosh}(\frac{e^{t-s}+e^{s-t}}{2})=C_N|t-s|$, so $\gamma_x$ has constant speed equal to ${C_N}$.\\
If $G$ is an upper gradient of $F$, setting $\widetilde G(x,t)=G(x,e^t)$, we get

\[
|\widetilde F(x,t)-\widetilde F(x,\tau)|
\le {C_N}
\int_t^\tau
\widetilde G(x,s)\,ds \qquad \forall  t,\tau \in I.
\]
 In particular $C_N \widetilde G$ is an upper gradient for the one variable function $\widetilde{F}(x,\cdot)$.
The one-dimensional global Poincar\'e inequality on the interval $I$ therefore gives
\[
\int_I
|\widetilde F(x,t)-\widetilde F_I(x)|^p\,\mathrm{d}t
\le
Cr^p
\int_I
\widetilde G(x,t)^p\,\mathrm{d}t.
\]
Integrating in $x$ yields

\[
I_1
\le
Cr^p
\int_Q
\widetilde{G}^p\,\mathrm{d}\mu_N\mathrm{d}t.
\]
\medskip
\noindent
We now estimate $I_2$. Since $\widetilde F_I$ depends only on $x$ and \begin{align*}
    (\widetilde F_I)_Q&=\frac{1}{|I|\mu_N(B_N(x_0,R))}\int_I \int_{B_N(x_0,R)} \widetilde F_I (x) \mathrm{d}\mu_N \mathrm{d}t= (\widetilde F_I)_{B_N(x_0,R)}
\end{align*}
we get that
\[
I_2
=
|I|
\int_{B_N(x_0,R)}
|\widetilde F_I-(\widetilde F_I)_{B_N(x_0,R)}|^p
\,\mathrm{d}\mu_N.
\]
Applying the global Poincar\'e inequality on $N$ gives

\[
I_2
\le
C
|I|
R^p
\int_{B_N(x_0,\lambda R)}
g_I^p\,\mathrm{d}\mu_N ,
\]
if $g_I$ is any upper gradient of $\widetilde F_I$.\\
We claim that
\begin{align}\label{defgI}
g_I(x)
=
\frac1{|I|}
\int_I
\frac{C_N\widetilde G(x,t)}{e^t}
\,\mathrm{d}t
\end{align}
is an upper gradient of $\widetilde F_I$.

We first estimate the metric speed of horizontal curves. Let $\gamma:[0,L]\to N$ be a rectifiable curve parametrized by arc length
with respect to $d_N$ joining $x$ and $y$, and fix $t\in\mathbb R$.
Define its horizontal lift
\[
\Gamma(s)=(\gamma(s),e^t) \qquad \forall s\in[0,L].
\]
For every $h\neq0$,
\[
\begin{aligned}
d_S(\Gamma(s+h),\Gamma(s))
&=
2C_N
\operatorname{arsinh}
\!\left(
\frac{
d_N(\gamma(s+h),\gamma(s))
}{
2e^t
}
\right)  \\
&\le
\frac{C_N}{e^t}
d_N(\gamma(s+h),\gamma(s)),
\end{aligned}
\]
since $\mathrm{arcosh}(1+2u^2)=2\,\mathrm{arsinh}(u)$ and $\mathrm{arsinh}(u)\le u$ for every $u\ge0$.
Dividing by $|h|$ and letting $h\to0$, we obtain
\[
|\Gamma'|_S(s)
\le
\frac{C_N}{e^t}
|\gamma'|_N(s)
=
\frac{C_N}{e^t}
\]
for almost every $s\in[0,L]$, because $\gamma$ is parametrized by arc length and $|\Gamma'|_S$ and $|\gamma'|_N$ denote the metric speeds of $\Gamma$ and $\gamma$, respectively. Therefore,
\[
\begin{aligned}
\int_\Gamma G\,\mathrm{d}s_S
&=
\int_0^L
\widetilde G(\gamma(s),t)\,
|\Gamma'|_S(s)\,\mathrm{d}s   \\
&\le
\frac{C_N}{e^t}
\int_0^L
\widetilde G(\gamma(s),t)\,\mathrm{d}s  \\
&=
\frac{C_N}{e^t}
\int_\gamma
\widetilde G(\xi,t)\,\mathrm{d}s_N(\xi).
\end{aligned}\]
Then, since $\Gamma$ joins $(x,e^t)$ and $(y,e^t)$, the upper gradient property of $G$ on $S$ gives $|F(x,e^t)-F(y,e^t)| \le \int_{\Gamma} G \mathrm{d}s_S$. Hence, an application of Tonelli's theorem yields
\[
|\widetilde F_I(x)-\widetilde F_I(y)| \le \frac{1}{|I|}\int_I \int_\Gamma G \mathrm{d}s_S\le \int_\gamma \frac{1}{|I|}\int_I\frac{C_N\widetilde G(\xi,t)}{e^t} \mathrm{d}t\mathrm{d}s_N(\xi).
\]
This shows that $g_I$ defined in \eqref{defgI} is an upper gradient of $\widetilde F_I$. \\
Since $|t-t_0|\le r$ and $r\le r_0$, $e^t \simeq_{r_0} a_0$,
uniformly on $I$. Therefore
\[
g_I(x)
\le
\frac{C}{a_0}
\frac1{|I|}
\int_I
\widetilde G(x,t)\,\mathrm{d}t,
\]
where $C$ depends on $r_0$. By Jensen's inequality,
\[
g_I(x)^p
\le
\frac{C}{a_0^p}
\frac1{|I|}
\int_I
\widetilde G(x,t)^p\,\mathrm{d}t.
\]
Consequently,
\[
I_2
\le
C
|I|
R^p
\frac1{a_0^p}
\frac1{|I|}
\int_{\lambda Q}
\widetilde G^p\,\mathrm{d}\mu_N\mathrm{d}t\le
Cr^p
\int_{\lambda Q}
\widetilde G^p\,\mathrm{d}\mu_N\mathrm{d}t.
\]
where $\lambda Q= B_N(x_0,\lambda R) \times I$ and we have used that $R=a_0r.$

Combining the estimates for $I_1$ and $I_2$ we find
\begin{align}\label{Poinc on cyl}
\int_Q
|\widetilde F-\widetilde F_Q|^p
\,\mathrm{d}\mu_N\mathrm{d}t
\le
Cr^p
\int_{\lambda Q}
\widetilde G^p\,\mathrm{d}\mu_N\mathrm{d}t.
\end{align}
Finally, using  \eqref{rombi}
and the comparability of the corresponding measures,
a standard averaging argument yields \eqref{lemma46} as follows.
For every constant $c\in \mathbb R$
\[
|F-F_{B_S(z_0,r)}|
\le |F-c|+|c-F_{B_S(z_0,r)}|,
\]
while
\begin{align*}
|c-F_{B_S(z_0,r)}|
&=
\left|
\frac1{\rho(B_S(z_0,r))}
\int_{B_S(z_0,r)}(F-c)\,\mathrm{d}\rho
\right|
\\&\le
\left(
\frac1{\rho(B_S(z_0,r))}
\int_{B_S(z_0,r)}|F-c|^p\,\mathrm{d}\rho
\right)^{1/p}
\end{align*}
by Hölder's inequality. Hence
\[
\int_{B_S(z_0,r)}|c-F_{B_S(z_0,r)}|^p\,\mathrm{d}\rho
\le
\int_{B_S(z_0,r)}|F-c|^p\,\mathrm{d}\rho,
\]
and therefore
\[
\int_{B_S(z_0,r)}|F-F_{B_S(z_0,r)}|^p\,\mathrm{d}\rho
\le
2^p
\int_{B_S(z_0,r)}|F-c|^p\,\mathrm{d}\rho.
\]
Choosing $c=F_{Q_{c_2r}(z_0)}$, \eqref{rombi} implies

\[
\int_{B_S(z_0,r)}|F-F_{B_S(z_0,r)}|^p\,\mathrm{d}\rho
\le
C_p
\int_{Q_{c_2r}(z_0)}
|F-F_{Q_{c_2r}(z_0)}|^p\,\mathrm{d}\rho.
\]

Applying the Poincaré inequality \eqref{Poinc on cyl} on the cylinder
$Q_{c_2r}(z_0)$, we obtain
\[
\int_{B_S(z_0,r)}|F-F_{B_S(z_0,r)}|^p\,\mathrm{d}\rho
\le
Cr^p
\int_{\lambda Q_{c_2r}(z_0)} G^p\,\mathrm{d}\rho,
\]
where $\lambda Q_{c_2r}(z_0)=B(x_0,\lambda a_0c_2 r) \times (t_0-c_2r,t_0+c_2r)$. Finally, by Lemma \ref{lem:ball-measure}, there exists a constant $\lambda'>0$ depending on $r_0$ such that
\[
\lambda Q_{c_2r}(z_0)\subset B_S(z_0,\lambda'r),
\]
and therefore
\[
\int_{B_S(z_0,r)}
|F-F_{B_S(z_0,r)}|^p\,\mathrm{d}\rho
\le
Cr^p
\int_{B_S(z_0,\lambda'r)}G^p\,\mathrm{d}\rho.
\]
This concludes the proof.
\end{proof}
\begin{proof}[Proof of Theorem \ref{main:theorem}]
By Corollary \ref{cor:quasi-flow-doubling}, the measure $\mu$ on $\mathcal{SW}_N$ defined in \eqref{defmu} is a locally doubling quasi-flow on the spiderweb $\mathcal{SW}_N$.
Hence, by Theorem~\ref{Proposition Poinc on SW}, the graph $(\mathcal{SW}_N,d_{\mathcal{SW}_N},\mu)$ supports a
global $L^p$-Poincaré inequality.

By Theorem~\ref{thm:1QI}, the graph $\mathcal{SW}_N$ is $(1,C)$-quasi-isometric to
$(S,d_S)$ and its vertices form a uniformly dense and uniformly
separated subset of $S$. Moreover, by construction, the measures
$\mu$ and $\rho$ are compatible in the sense of Definition~\ref{def:comp}.

Finally, Lemma~\ref{lem:localPI} shows that $(S,d_S,\rho)$ supports a weak local
$L^p$-Poincaré inequality up to every prescribed scale. In particular,
it holds up to a scale larger than $2D+k_2$, where $D$ and $k_2$ are
the constants appearing in Theorem~\ref{transference}.

Therefore all the assumptions of Theorem~\ref{transference} are satisfied, and we
conclude that $(S,d_S,\rho)$ supports a global $L^p$-Poincaré
inequality.
\end{proof}

\end{document}